\documentclass[]{article}
\usepackage[margin=1.1in]{geometry}
\usepackage{scrextend}
\usepackage{amsfonts,nicefrac,bm}
\usepackage[hidelinks]{hyperref}
\usepackage{natbib}
\usepackage[english]{babel}
\usepackage{times}
\usepackage[T1]{fontenc}
\usepackage{bbm}
\usepackage{color,xcolor}
\usepackage{xspace}
\usepackage[normalem]{ulem}
\usepackage{rotating}
\usepackage{booktabs}
\usepackage{makecell}

\usepackage{amssymb}
\usepackage[utf8]{inputenc}
\usepackage{tensor}
\usepackage{wrapfig}
\usepackage{subcaption}
\usepackage[export]{adjustbox}
\usepackage{floatrow}
\usepackage{comment}
\usepackage{enumitem}
\usepackage{algorithmic}
\usepackage[ruled, lined, commentsnumbered, longend]{algorithm2e}
\usepackage{stackengine}
\usepackage{array}
\usepackage{graphicx,subcaption}
\usepackage{multirow}
\usepackage{floatrow}

\makeatletter 
\g@addto@macro{\@algocf@init}{\SetKwInOut{Note}{Note}} 
\makeatother

\newcommand{\indep}{\perp \!\!\! \perp}

\usepackage{amsmath}
\usepackage{amsthm}

\theoremstyle{plain}
\newcounter{example}
\newtheorem{theorem}{Theorem}
\newtheorem{lemma}{Lemma}[section]

\theoremstyle{definition}
\newtheorem{definition}{Definition}

\theoremstyle{remark}
\newtheorem{remark}{Remark}[section]

\newcommand{\eqd}{\stackrel{\textnormal{d}}{=}} 
\newcommand{\stkout}[1]{\ifmmode\text{\sout{\ensuremath{#1}}}\else\sout{#1}\fi}

\title{Distribution-Free Conditional Median Inference}
\author{
  Dhruv Medarametla\thanks{Department of Statistics, Stanford University}
  \and
  Emmanuel Candès\thanks{Department of Statistics and Mathematics, Stanford University}
}
\date{September 1, 2021}

\begin{document}

\maketitle

\begin{abstract}
We consider the problem of constructing confidence intervals for the
median of a response $Y \in \mathbb{R}$ conditional on features
$X \in \mathbb{R}^d$ in a situation where we are not willing to
make any assumption whatsoever on the underlying distribution of the
data $(X,Y)$.  We propose a method based upon ideas from conformal
prediction and establish a theoretical guarantee of coverage while
also going over particular distributions where its performance is
sharp. 
Additionally, we prove an equivalence between confidence intervals for the conditional median and confidence intervals for the response variable, resulting in a lower bound on the length of any
possible conditional median confidence interval. 
This lower bound is independent of sample size and holds for all distributions with no
point masses.
\end{abstract}

\section{Introduction}

Consider a dataset
$(X_1,Y_1),\ldots,(X_n,Y_n)\subseteq \mathbb{R}^d\times\mathbb{R}$ and
a test point $(X_{n+1},Y_{n+1})$, with all datapoints being drawn
i.i.d.~from the same distribution $P$. Given our training data, can we
provide a confidence interval for the expected value $\mu(X_{n+1}) = \mathbb{E}[Y_{n+1}|X_{n+1}]$?

Methods for inferring the conditional mean are
certainly not in short supply. 
In fact, most existing methods predict not just the conditional mean at a datapoint $\mathbb{E}[Y_{n+1}|X_{n+1}]$, but the full conditional mean function $\mathbb{E}[Y|X=x]$ for all $x\in\mathbb{R}^d$. 
To the best of our knowledge, however,
each approach relies on some assumptions in order to guarantee
coverage. For instance, the classical linear regression model is often used but
is only accurate if $Y|X$ is normal with mean
$\mu(x) = \mathbb{E}[Y|X=x]$ affine in $x$ and standard deviation
independent of $x$. Nonparametric regressions cannot estimate the
conditional mean without imposing smoothness conditions and assuming
that the conditional distribution is sufficiently light tailed.
Since reliable conditional mean inference is a very common problem,
these methods are nevertheless used all the time, e.g.~in predicting
disease survival times, classifying spam emails, pricing financial
assets, and more. The issue is that the assumptions these methods make
rarely hold in practice. Thus, the question remains:
is it possible to estimate the conditional mean at a datapoint in a
\textit{distribution-free} setting, with no assumptions on $P$?

It turns out that it is not only impossible to get a nontrivial
confidence interval for the conditional mean $\mathbb{E}[Y|X=X_{n+1}]$, but it
is actually impossible to get a confidence interval for
$\mathbb{E}[Y]$ itself. This result originates in
\cite{bahadur1956nonexistence}, where the authors show that any
parameter sensitive to tails of a distribution cannot be estimated
when no restrictions exist on the distribution class; an example of a
distribution with a non-estimable mean is given in Appendix
\ref{Conditional Mean Proof}.

Thus, within the distribution-free setting, making progress on
inferring the conditional mean requires a modification to the problem
statement. One strategy is to restrict the range of $Y$. An example of
this is in \cite{barber2020distribution}, which introduces an
algorithm that calculates a confidence interval for the conditional
mean in the case where $Y\in\{0,1\}$. However, even with this
restriction, Barber shows that there exists a fundamental bound
limiting how small any such confidence interval can be.

The other strategy is to modify the measure of central tendency that
we study. Bahadur and Savage's result suggests that the best
parameters to study are robust to distribution outliers; this
observation motivates our investigation of the conditional median.

In a nutshell, the conditional median is possible to infer because of
the strong and quantifiable relationship between any particular
sampled datapoint $(X_i,Y_i)$ and $\textrm{Median}(Y|X=X_i)$.
Its robustness
to outliers means that even within the distribution-free setting,
there is no need to worry about `hidden' parts of the distribution.
Additionally, there already exists a well-known algorithm for estimating
$\textrm{Median}(Y)$ given a finite number $n$ of
i.i.d.~samples. Explored in \cite{noether1972distribution} and covered
in Appendix \ref{Median Proof}, this algorithm produces intervals with guaranteed rates of coverage and
widths going to zero as the sample size goes to infinity, suggesting
that an algorithm for the conditional median might also perform well. We note that there already exists literature dealing with estimating the conditional quantile through quantile regression; \cite{koenker2005} provides an introduction to different methods. The difference between quantile regression and our work here is that quantile regression requires continuity conditions and only provides theoretical guarantees at the asymptotic level, as seen in \cite{belloni2019conditional}, \cite{takeuchi2006nonparametric}, and \cite{oberhofer2016asymptotic}; the methods described here provide coverage guarantees on finite sample sizes and require no assumptions. 

Our goal in this paper is to combine ideas from regular median
inference with procedures from distribution-free inference in order to
understand how well an algorithm can cover the conditional median and,
more generally, conditional quantiles. In particular, we want to see
if the properties of the median and quantiles lead to a valid
inference method while also examining the limits of this inference.

It is important to note that the quantity we are attempting to study is $\textrm{Median}(Y|X = X_{n+1})$, not $\textrm{Median}(Y|X = x)$. The first term is a random variable dependent on $X_{n+1}$, whereas the second is fixed and exists for all $x$ in the support of $P$. We focus on the first quantity because we are in the distribution-free setting; as we know nothing about the class of distributions that $P$ belongs to, it is more tractable to make inferences about datapoints as opposed to pointwise across the full distribution. 

Another way to frame this is to think of our goal as to cover the value of the conditional median function when weighting by the marginal distribution $P_X$, as opposed to covering the full conditional median function over all $x\in\mathbb{R}^d$.
Because we are predicting the conditional median at an unknown value $X_{n+1}$, our success is not measured by whether or not we predict the conditional median correctly at all possible $x\in\mathbb{R}^d$, but by how often we predict the conditional median correctly across $P_X$.

The methods used in this paper are similar to those from
distribution-free \textit{predictive inference}, which focuses on
predicting $Y_{n+1}$ from a finite training set. The field of
conformal predictive inference began with \cite{vovk2005algorithmic}
and was built up by works such as \cite{shafer2008tutorial} and
\cite{vovk2009line}; it has been generating interest recently due to
its versatility and lack of assumptions. Applications of conformal
predictive inference range from criminal justice to drug discovery, as
seen in \cite{Romano2020With} and \cite{cortes2019concepts}
respectively. 

While this paper relies on techniques from predictive inference, our
focus is on \textit{parameter inference}, which is quite different
from prediction because it focuses on predicting a function of the conditional distribution $Y_{n+1}|X_{n+1}$ as opposed to the datapoint $Y_{n+1}$. For example, using sample datapoints from an unknown normal distribution to estimate the true mean of the distribution would constitute parameter inference; using the estimated quantiles of the datapoints to create a predictive confidence interval for the next datapoint would constitute predictive inference. Whereas predictive inference exploits the fact that
$(X_{n+1},Y_{n+1})$ is exchangeable with the sample datapoints,
parameter inference requires another layer of analysis, as $(X_{n+1},\textrm{Median}(Y_{n+1}|X_{n+1}))$ is not exchangeable with $(X_i, Y_i)$. This additional complexity
demands modifying approaches from predictive inference to produce
valid {\em parameter} inference.

\subsection{Terminology}

We begin by setting up definitions to formalize the concepts above. Throughout this paper, we assume that any distribution $(X,Y)\sim P$ is over $\mathbb{R}^d\times\mathbb{R}$ unless explicitly stated otherwise. Each result in this paper holds true for all values of $d$ and $n$.

Given a feature vector $X_{n+1}$, we let
$\hat{C}_n(X_{n+1})\subseteq\mathbb{R}$ denote a confidence interval for
some functional of the conditional distribution
$Y_{n+1}|X_{n+1}$. Note that we use the phrase confidence interval
for convenience; in its most general form, $\hat{C}_n(X_{n+1})$ is a subset
of $\mathbb{R}$. This interval is a function of the point $X_{n+1}$ at which
we seek inference as well as of our training data
$\mathcal{D} = \{(X_1,Y_1),\ldots,(X_n,Y_n)\}$. We write $\hat{C}_n$
to refer to the general algorithm that maps $\mathcal{D}$ to the
resulting confidence intervals $\hat{C}_n(x)$ for each
$x\in\mathbb{R}^d$.

In order for $\hat{C}_n$ to be useful, we want it to \textit{capture},
or contain, the parameter we care about with high probability. We
formalize this as follows:
\begin{definition}
\label{Median Coverage Definition}
We say that $\hat{C}_n$ satisfies \textit{distribution-free \textbf{median} coverage} at level $1-\alpha$, denoted by $(1-\alpha)$-Median, if 
$$\mathbb{P}\big\{\textrm{Median}(Y_{n+1}|X_{n+1})\in\hat{C}_n(X_{n+1}) \big\}\geq 1-\alpha\textnormal{ for all distributions }P\textnormal{ on }(X,Y)\in\mathbb{R}^d\times\mathbb{R}.$$
\end{definition}
\begin{definition}
\label{Quantile Coverage Definition}
For $0<q<1$, let $\textrm{Quantile}_q(Y_{n+1}|X_{n+1})$ refer to the $q$th quantile of the conditional distribution $Y_{n+1}|X_{n+1}$.
We say that $\hat{C}_n$ satisfies \textit{distribution-free \textbf{quantile} coverage} for the $q$th quantile at level $1-\alpha$, denoted by $(1-\alpha,q)$-Quantile, if 
$$\mathbb{P}\{\textrm{Quantile}_q(Y_{n+1}|X_{n+1})\in\hat{C}_n(X_{n+1}) \}\geq 1-\alpha\textnormal{ for all distributions }P\textnormal{ on }(X,Y)\in\mathbb{R}^d\times\mathbb{R}.$$
\end{definition}

The probabilities in Definitions \ref{Median Coverage Definition} and \ref{Quantile Coverage Definition} are both taken over the training data $\mathcal{D} = \{(X_1,Y_1),\ldots,(X_n,Y_n)\}$ and test point $X_{n+1}$. Thus, satisfying $(1-\alpha)$-Median is equivalent to satisfying $(1-\alpha,0.5)$-Quantile.

These concepts are similar to predictive coverage, with the key difference being that our goal is now to predict a function of $X_{n+1}$ rather than a new datapoint $Y_{n+1}$.

\begin{definition}
\label{Predictive Coverage Definition}
We say that $\hat{C}_n$ satisfies \textit{distribution-free \textbf{predictive} coverage} at level $1-\alpha$, denoted by $(1-\alpha)$-Predictive, if 
$$\mathbb{P}\big\{Y_{n+1}\in\hat{C}_n(X_{n+1}) \big\}\geq 1-\alpha\textnormal{ for all distributions }P\textnormal{ on }(X,Y)\in\mathbb{R}^d\times\mathbb{R},$$
where this probability is taken over the training data $\mathcal{D} = \{(X_1,Y_1),\ldots,(X_n,Y_n)\}$ and test point $(X_{n+1},Y_{n+1})$.
\end{definition}

Finally, we define the type of conformity scores that we will use in our general quantile inference algorithm.

\begin{definition}
\label{Locally Nondecreasing Conformity Score Definition}
We say that a function $f:(\mathbb{R}^d,\mathbb{R})\to\mathbb{R}$ is a \textit{locally nondecreasing conformity score} if, for all $x\in\mathbb{R}^d$ and $y, y'\in\mathbb{R}$ with $y\leq y'$, we have $f(x,y)\leq f(x,y')$.

\end{definition}

\subsection{Summary of Results}

We find that there exists a distribution-free predictive inference algorithm that satisfies both $(1-\alpha/2)$-Predictive and $(1-\alpha)$-Median. Moreover, an improved version of this algorithm also satisfies $(1-\alpha,q)$-Quantile. Together, these prove that there exists nontrivial algorithms $\hat{C}_n$ that satisfy $(1-\alpha)$-Median and $(1-\alpha,q)$-Quantile for all $0<q<1$. 

We go on to show that conditional median inference and predictive inference are nearly equivalent problems.
Specifically, we show that any algorithm that contains $\textrm{Median}(Y_{n+1}|X_{n+1})$ with probability $1-\alpha$ must also contain $Y_{n+1}$ with probability at least $1-\alpha$, and that any algorithm that contains $Y_{n+1}$ with probability at least $1-\alpha/2$ must contain $\textrm{Median}(Y_{n+1}|X_{n+1})$ with probability $1-\alpha$.

Taken together, these results give us somewhat conflicting
perspectives. On the one hand, there exist distribution-free
algorithms that capture the conditional median and conditional
quantile with high likelihood; on the other hand, any such algorithm
will also capture a large proportion of the distribution itself,
putting a hard limit on how well such algorithms can ever perform.

\section{Confidence Intervals for the Conditional Median}
\label{Confidence Interval for the Conditional Median Section}

This section proves the existence of algorithms obeying
distribution-free median and quantile coverage.  We then focus on
situations where these algorithms are sharp.

\subsection{Basic Conditional Median Inference}
\label{Conditional Median Algorithm Section}

Algorithm \ref{Conditional Median Interval} below operates by taking
the training dataset and separating it into two halves of sizes
$n_1+n_2 = n$.  Next, a regression algorithm $\hat{\mu}$ is trained on
$\mathcal{D}_1 = \{(X_1,Y_1),\ldots,(X_{n_1},Y_{n_1})\}$. The
residuals $Y_i - \hat{\mu}(X_i)$ are calculated for $n_1<i\leq n$, and
the $1-\alpha/2$ quantile of the absolute value of these residuals is
then used to create a confidence band around the prediction
$\hat{\mu}(X_{n+1})$. The expert will recognize that this is identical
to a well-known algorithm from predictive inference as explained
later.

\begin{algorithm}[H]
\caption{Confidence Interval for $\textrm{Median}(Y_{n+1}|X_{n+1})$ with Coverage $1-\alpha$}
\label{Conditional Median Interval}
\SetKwInOut*{KwIn}{}
\SetKwInOut{KwProcess}{Process}
\SetKwInOut*{KwOut}{}
\KwIn{}
\begin{algorithmic}
\STATE Number of i.i.d. datapoints $n\in\mathbb{N}$.
\STATE Split sizes $n_1+n_2 = n$.
\STATE Datapoints $(X_1,Y_1),\ldots,(X_n,Y_n)\sim P\subseteq(\mathbb{R}^d,\mathbb{R})$.
\STATE Test point $X_{n+1}\sim P$.
\STATE Regression algorithm $\hat{\mu}$.
\STATE Coverage level $1-\alpha\in (0,1)$.
\end{algorithmic}

\KwProcess{}
\begin{algorithmic}
\STATE Randomly split $\{1,\ldots,n\}$ into disjoint $\mathcal{I}_1$ and $\mathcal{I}_2$ with $|\mathcal{I}_1| = n_1$ and $|\mathcal{I}_2| = n_2$.
\STATE Fit regression function $\hat{\mu}:\mathbb{R}^d\to \mathbb{R}$ on $\{(X_i,Y_i): i\in \mathcal{I}_1\}$.
\STATE For $i\in \mathcal{I}_2$ set $E_i = |Y_i - \hat{\mu}(X_i)|$.
\STATE Compute $Q_{1-\alpha/2}(E)$, the $(1-\alpha/2)(1+1/n_2)$-th empirical quantile of $\{E_i:i\in\mathcal{I}_2\}$.
\end{algorithmic}

\KwOut{}
\begin{algorithmic}
\STATE Confidence interval $\hat{C}_n(X_{n+1}) = [\hat{\mu}(X_{n+1})-Q_{1-\alpha/2}(E),\hat{\mu}(X_{n+1})+Q_{1-\alpha/2}(E)]$ for \\
$\textrm{Median}(Y_{n+1}|X_{n+1})$.
\end{algorithmic}
\end{algorithm}

\begin{theorem}
\label{Conditional Median Theorem}
For all distributions $P$, all regression algorithms $\hat{\mu}$, and all split sizes $n_1 + n_2 = n$, the output of Algorithm \ref{Conditional Median Interval} contains $\textrm{Median}(Y_{n+1}|X_{n+1})$ with probability at least $1-\alpha$. That is, the algorithm satisfies $(1-\alpha)$-Median. 
\end{theorem}

The proof of Theorem \ref{Conditional Median Theorem} is covered in Appendix \ref{Conditional Median Proof}.

\begin{remark}
Algorithm \ref{Conditional Median Interval} works independently of how $\hat{\mu}$ is trained; this means that any regression function may be used, from simple linear regression to more complicated machine learning algorithms. It is important to note that $\hat{\mu}(x)$ does not need to be an estimate of the true conditional mean $\mu(x)$; while one option is to train it to predict the conditional mean, it can be fit to predict the conditional median, conditional quantile, or any other measure of central tendency. The best choice for what to fit $\hat{\mu}$ to may depend on one's underlying belief about the distribution.
\end{remark}

We return at last to the connection with predictive inference.
Introduced in \cite{papadopoulos2002inductive} and
\cite{vovk2005algorithmic} and studied in \cite{lei2013distribution},
\cite{foygel2021limits}, \cite{Romano2020With}, and several other
papers, the \textit{split conformal method} was initially created to
achieve distribution-free predictive coverage guarantees. In
particular, \cite{vovk2005algorithmic} shows that Algorithm
\ref{Conditional Median Interval} satisfies $(1-\alpha/2)$-Predictive,
implying that in order to capture $\textrm{Median}(Y_{n+1}|X_{n+1})$
with probability $1-\alpha$, our algorithm produces a wider confidence
interval than an algorithm trying to capture $Y_{n+1}$ with the same
probability.

\subsection{General Conditional Quantile Inference}
\label{Conditional Quantile Algorithm Section}

Algorithm \ref{Conditional Median Interval} is a good first step
towards a usable method for conditional median inference; however, it
may be too rudimentary to be used in practice. Algorithm
\ref{Conditional Quantile Interval} is a more general version of Algorithm \ref{Conditional Median Interval} that results in conditional quantile coverage and better empirical performance.
This provides a better understanding of how diverse parameter inference algorithms can be.

\begin{algorithm}[H]
\caption{Confidence Interval for $\textrm{Quantile}_q(Y_{n+1}|X_{n+1})$ with Coverage $1-\alpha$}
\label{Conditional Quantile Interval}
\SetKwInOut*{KwIn}{}
\SetKwInOut{KwProcess}{Process}
\SetKwInOut*{KwOut}{}
\SetKwInput{KwData}{Note}

\KwIn{}

\begin{algorithmic}
\STATE Number of i.i.d.~datapoints $n\in\mathbb{N}$.
\STATE Split sizes $n_1+n_2 = n$.
\STATE Datapoints $(X_1,Y_1),\ldots,(X_n,Y_n)\sim P\subseteq(\mathbb{R}^d,\mathbb{R})$.
\STATE Test point $X_{n+1}\sim P$.
\STATE Locally nondecreasing conformity score algorithms $f^{\textnormal{lo}}$ and $f^{\textnormal{hi}}$.
\STATE Quantile level $q\in (0,1)$.
\STATE Coverage level $1-\alpha\in (0,1)$.
\STATE Split probabilities $r+s=\alpha$.
\end{algorithmic}

\KwProcess{}
\begin{algorithmic}
\STATE Randomly split $\{1,\ldots,n\}$ into disjoint $\mathcal{I}_1$ and $\mathcal{I}_2$ with $|\mathcal{I}_1| = n_1$ and $|\mathcal{I}_2| = n_2$.
\STATE Fit conformity scores $f^{\textnormal{lo}},f^{\textnormal{hi}}:(\mathbb{R}^d,\mathbb{R})\to \mathbb{R}$ on $\{(X_i,Y_i): i\in \mathcal{I}_1\}$.
\STATE For $i\in \mathcal{I}_2$ set $E_i^{\textnormal{lo}} = f^{\textnormal{lo}}(X_i,Y_i))$ and $E_i^{\textnormal{hi}} = f^{\textnormal{hi}}(X_i,Y_i))$.
\STATE Compute $Q_{rq}^{\textrm{lo}}(E)$, the $rq(1+1/n_2)-1/n_2$ empirical quantile of $\{E_i^{\textnormal{lo}}:i\in\mathcal{I}_2\}$, and $Q_{1-s(1-q)}^{\textrm{hi}}(E)$, the $(1-s(1-q))(1+1/n_2)$ empirical quantile of $\{E_i^{\textnormal{hi}}:i\in\mathcal{I}_2\}$.
\end{algorithmic}

\KwOut{}
\begin{algorithmic}
\STATE Confidence interval $\hat{C}_n(X_{n+1}) = \{y:Q_{rq}^{\textrm{lo}}(E)\leq f^{\textnormal{lo}}(X_{n+1},y),f^{\textnormal{hi}}(X_{n+1},y)\leq Q_{1-s(1-q)}^{\textrm{hi}}(E)\}$ for $\textrm{Quantile}_q(Y_{n+1}|X_{n+1})$.
\end{algorithmic}
\end{algorithm}

Algorithm \ref{Conditional Quantile Interval} differs from Algorithm \ref{Conditional Median Interval} in two ways. 
First, we use the $rq$ quantile of the lower scores to create the confidence interval's lower bound, and the $1-s(1-q)$ quantile of the upper scores (corresponding to the top $s(1-q)$ of the score distribution) for the upper bound. Second, the functions we fit are no longer regression functions, but instead locally nondecreasing conformity scores. These scores are described in Definition \ref{Locally Nondecreasing Conformity Score Definition}; see Remark \ref{Conformity Score Examples} for examples.

\begin{theorem}
\label{Conditional Quantile Theorem}
For all distributions $P$, all locally nondecreasing conformity scores $f^{\textnormal{lo}}$ and $f^{\textnormal{hi}}$, all split sizes $n_1 + n_2 = n$, and all $0<q<1$, the output of Algorithm \ref{Conditional Quantile Interval} contains $\textrm{Quantile}_q(Y_{n+1}|X_{n+1})$ with probability at least $1-\alpha$. That is, Algorithm \ref{Conditional Quantile Interval} satisfies $(1-\alpha,q)$-Quantile. 
\end{theorem}

The proof of Theorem \ref{Conditional Quantile Theorem} is covered in
Appendix \ref{Conditional Quantile Proof} and is similar to that of
Theorem \ref{Conditional Median Theorem}; the main modifications come
from the changes described above. Regarding the first change, the
asymmetrical quantiles on the lower and upper end of the $E_i$'s
balance the fact that datapoints have asymmetrical probabilities of
being on either side of the conditional quantile. Regarding the second
change, because the conformity scores $E_i$ still preserve relative
ordering, they do not affect the relationship between datapoints and
the conditional quantile.

\begin{remark}
  One possible choice for $r$ and $s$ is $r=s=\alpha/2$. This is
  motivated by the logic that $r$ and $s$ decide the probabilities of
  failure on the lower bound and the upper bound, respectively; if we
  want the bound to be equally accurate on both ends, it makes sense
  to set $r$ and $s$ equal. Another choice is $r=(1-q)\alpha$ and
  $s=q\alpha$; this results in the quantiles for
  $Q_{rq}^{\textrm{lo}}$ and $Q_{1-s(1-q)}^{\textrm{hi}}$ being
  approximately equal, with the algorithm taking the $q(1-q)\alpha$
  quantile of the scores on both the lower and upper ends. 
\end{remark}

\begin{remark}
\label{Conformity Score Examples}
The versatility of the conformity scores $f^{\textnormal{lo}}$ and $f^{\textnormal{hi}}$ is what differentiates Algorithm \ref{Conditional Quantile Interval} from Algorithm \ref{Conditional Median Interval} and makes it a viable option for conditional quantile inference. Below are a few examples of possible scores and the style of intervals they produce.

\begin{itemize}
    \itemsep-0.2em 
  \item[-]
    $f^{\textnormal{lo}}(X_i,Y_i) = f^{\textnormal{hi}}(X_i,Y_i) = Y_i
    - \hat{\mu}(X_i)$,
    where $\hat{\mu}:\mathbb{R}^d\to\mathbb{R}$ is trained on
    $\{(X_i,Y_i): i\in \mathcal{I}_1\}$ as a central tendency
    estimator. This is the conformity score used in Algorithm
    \ref{Conditional Median Interval} and \cite{vovk2005algorithmic},
    resulting in a confidence interval of the form
    $[\hat{\mu}(X_{n+1})+c_{\textnormal{lo}},\hat{\mu}(X_{n+1})+c_{\textnormal{hi}}]$
    for some
    $c_{\textnormal{lo}},c_{\textnormal{hi}}\in\mathbb{R}$. This score
    is best when the conditional distribution $Y|X=x$ is similar for
    all $x$ and either the mean or median can be estimated with
    reasonable accuracy. Note that if the conditional distribution
    $Y - \mathbb{E}[Y|X]$ is independent of $X$, Algorithm
    \ref{Conditional Quantile Interval} will output the same
    confidence interval for $\hat{\mu}(x) = \mathbb{E}[Y|X=x]$ and
    $\hat{\mu}(x) = \textrm{Median}(Y|X=x)$.
  \item [-]
    $f^{\textnormal{lo}}(X_i,Y_i) = f^{\textnormal{hi}}(X_i,Y_i) =
    \frac{Y_i - \hat{\mu}(X_i)}{\hat{\sigma}(X_i)}$, where
    $\hat{\mu}:\mathbb{R}^d\to\mathbb{R}$ and
    $\hat{\sigma}:\mathbb{R}^d\to\mathbb{R}^+$ are trained on
    $\{(X_i,Y_i): i\in \mathcal{I}_1\}$ as a central tendency
    estimator and conditional absolute deviation estimator,
    respectively. This score results in a confidence interval of the
    form
    $[\hat{\mu}(X_{n+1})+c_{\textnormal{lo}}\hat{\sigma}(X_{n+1}),\hat{\mu}(X_{n+1})+c_{\textnormal{hi}}\hat{\sigma}(X_{n+1})]$
    for some
    $c_{\textnormal{lo}},c_{\textnormal{hi}}\in\mathbb{R}$. Unlike the
    previous example, this score no longer results in a fixed-length
    confidence interval; it is best used when there is high
    heteroskedasticity in the underlying distribution. This is the
    conformity scores used to create adaptive predictive intervals in
    \cite{lei2018distribution}. Note that a normalization constant
    $\gamma>0$ can be added to the denominator $\hat{\sigma}(X_i)$ to
    create stable confidence intervals.
  \item[-]
    $f^{\textnormal{lo}}(X_i,Y_i) = Y_i -
    \hat{Q}^{\textnormal{lo}}(X_i)$ and
    $f^{\textnormal{hi}}(X_i,Y_i) = Y_i -
    \hat{Q}^{\textnormal{hi}}(X_i)$, where
    $\hat{Q}^{\textnormal{lo}},\hat{Q}^{\textnormal{hi}}:\mathbb{R}^d\to\mathbb{R}$
    are trained on $\{(X_i,Y_i): i\in \mathcal{I}_1\}$ to estimate the $rq$ quantile and the $1-s(1-q)$ quantile of the conditional distribution, respectively. This choice
    results in a confidence interval of the form
    $[\hat{Q}^{\textnormal{lo}}(X_{n+1})+c_{\textnormal{lo}},\hat{Q}^{\textnormal{hi}}(X_{n+1})+c_{\textnormal{hi}}]$
    for some
    $c_{\textnormal{lo}},c_{\textnormal{hi}}\in\mathbb{R}$. These 
    scores are best when one can estimate the conditional quantiles
    reasonably well and the conditional distribution $Y|X=x$ is
    heteroskedastic. Note that if $\hat{Q}^{\textnormal{lo}}$ and $\hat{Q}^{\textnormal{hi}}$ are trained well, then the resulting confidence interval will be approximately $[\hat{Q}^{\textnormal{lo}}(X_{n+1}),\hat{Q}^{\textnormal{hi}}(X_{n+1})]$. These are the scores used to create the
    predictive intervals seen in \cite{romano2019conformalized} and
    \cite{sesia2020comparison}.
  \item [-]
    $f^{\textnormal{lo}}(X_i,Y_i) = f^{\textnormal{hi}}(X_i,Y_i) =
    \hat{F}_{Y|X=X_i}(Y_i)$, where
    $\hat{F}_{Y|X=x}:\mathbb{R}^d\times\mathbb{R}\to[0,1]$ is trained
    on $\{(X_i,Y_i): i\in \mathcal{I}_1\}$ to be the estimated
    cumulative distribution function of the conditional distribution
    $Y|X$. Using this score will result in a confidence interval
    $[\hat{F}_{Y|X=X_{n+1}}^{-1}(c_{\textnormal{lo}}),\hat{F}_{Y|X=X_{n+1}}^{-1}(c_{\textnormal{hi}})]$
    for some $c_{\textnormal{lo}},c_{\textnormal{hi}}\in[0,1]$,
    similar to the predictive intervals in
    \cite{chernozhukov2019distributional} and
    \cite{kivaranovic2020adaptive}. This can be a good approach when
    the conditional distribution $Y|X$ is particularly complex.
    \item[-] $f^{\textnormal{lo}}(X_i,Y_i) =
      f^{\textnormal{hi}}(X_i,Y_i) = \log Y_i - \hat{\mu}(X_i)$, where
      $\hat{\mu}:\mathbb{R}^d\to\mathbb{R}$ is trained on
      $\{(X_i,Y_i): i\in \mathcal{I}_1\}$ as a log central tendency estimator. This results in a confidence interval of the form $[c_{\textnormal{lo}}\exp(\hat{\mu}(X_{n+1})),$ $c_{\textnormal{hi}}\exp(\hat{\mu}(X_{n+1}))]$ for some $c_{\textnormal{lo}},c_{\textnormal{hi}}\in\mathbb{R}^{+}$. This score works well when $Y$ is known to be positive and one wants to minimize the approximation ratio; it is equivalent to taking a $\log$ transformation of the data. 
\end{itemize}

In general, a good choice for $f^{\textnormal{lo}}(X_i,Y_i)$ and $f^{\textnormal{hi}}(X_i,Y_i)$ depends on one's underlying belief about the distribution as well as on the sample size $n$, though some scores perform better in practice. The choice of $n_1$ and $n_2$ represents a balance between model training and interval tightness; increasing $n_1$ increases the amount of data for $f^{\textnormal{lo}}$ and $f^{\textnormal{hi}}$, and increasing $n_2$ results in a better quantile for the predictive interval.  \cite{sesia2020comparison} contains more information on the effect of the conformity score on the size of predictive intervals as well as the impact of the ratio $n_1/n$ on interval width and coverage. We also simulate the impact of different scores on conditional quantile intervals in Section \ref{Simulations Section}.
\end{remark}

\begin{remark}
\label{Cross-Conformal Generalization}
Algorithm \ref{Conditional Quantile Interval} can be generalized to use more than one split, resulting in multiple confidence intervals that can then be combined into one output. For more information, \cite{vovk2018cross} describes how to apply $k$-fold cross validation to split conformal inference, and \cite{barber2021predictive} describes the special case of having $n$ different splits using the jackknife$+$ procedure. 
\end{remark}

\subsection{Algorithm Sharpness}
\label{Algorithm Sharpness Section}

Now that we have seen that Algorithms \ref{Conditional Median
  Interval} and \ref{Conditional Quantile Interval} achieve coverage,
an important question to ask is whether or not the terms for the error
quantile can be improved. Do our methods consistently overcover the
conditional median, and if so, is it possible to take a lower quantile
of the error terms and still have Theorems \ref{Conditional Median
  Theorem} and \ref{Conditional Quantile Theorem} hold?  In this
section, we prove that this is impossible by going over a particular
distribution $P^{\delta}$ for which the $1-\alpha/2$ term in Algorithm
\ref{Conditional Median Interval} is necessary. Additionally, we go
over a choice $\hat{\mu}_c$ with the property that Algorithm
\ref{Conditional Median Interval} \textit{always} results in
$1-\alpha/2$ coverage {when ran with input $\hat{\mu}_c$; this implies
  that there does not exist a distribution with the property that
  Algorithm \ref{Conditional Median Interval} will always provide a
  sharp confidence interval for the conditional median regardless of
  the regression algorithm.}

For each $\delta>0$, consider $(X,Y)\sim P^{\delta}$ over
$\mathbb{R}\times\mathbb{R}$, where
$P^{\delta}_X = \textrm{Unif}[-0.5,0.5]$ and $Y|X \eqd X \, B$;
$B \in \{0,1\}$ is here an independent Bernoulli variable with
$\mathbb{P}(B = 1) = 0.5+\delta$.  That is, $0.5+\delta$ of the distribution
is on the line segment $Y=X$ from $(-0.5,-0.5)$ to $(0.5,0.5)$, and
$0.5-\delta$ of the distribution is on the line segment $Y=0$ from
$(-0.5,0)$ to $(0.5,0)$. Thus, $\textrm{Median}(Y|X=x)=x$. A
visualization of $P^{\delta}$ is shown in Figure \ref{P^delta Figure}.

\begin{figure}[H]
\includegraphics[width=0.6\textwidth]{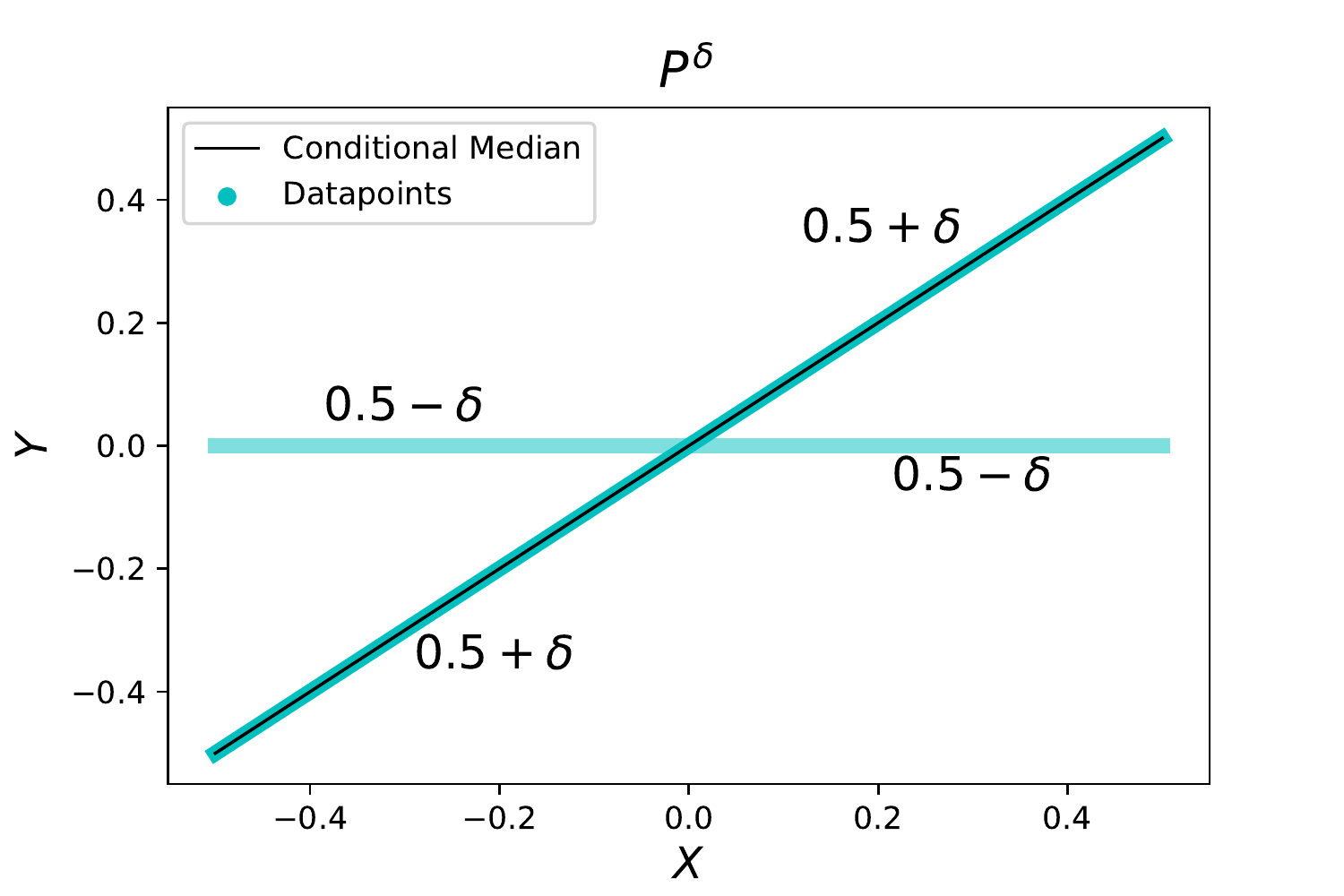}
\caption{A distribution for which it is difficult to estimate the
  conditional median. }
\label{P^delta Figure}
\end{figure}

We know that Algorithm \ref{Conditional Median Interval} is accurate
for all distributions $P$ and all algorithms $\hat{\mu}$. Consider the
regression algorithm $\hat{\mu}:\mathbb{R}\to\mathbb{R}$ such that
$\hat{\mu}(x)=0$ for all $x\in\mathbb{R}$; in other words, $\hat{\mu}$
predicts $Y_i=0$ for all $X_i$. We show that Algorithm
\ref{Conditional Median Interval} returns a coverage almost exactly
equal to $1-\alpha$. 

\begin{theorem}
\label{Sharpness Case Theorem}
For all $\epsilon>0$, there exist $N$ and $\delta>0$ such that if we
sample $n>N$ datapoints from the distribution $P^{\delta}$ and use
Algorithm \ref{Conditional Median Interval} with $\hat{\mu} = 0$ as
defined above and $n_1 = n_2 = n/2$ to get a confidence interval for
the conditional median,
$$\mathbb{P}\big\{\textnormal{Median}(Y_{n+1}|X_{n+1})\in\hat{C}_n(X_{n+1}) \big\}\leq 1-\alpha + \epsilon.$$
\end{theorem}

The proof is in Appendix \ref{Sharpness Case Proof}. Theorem
\ref{Sharpness Case Theorem} does not directly prove that the
$1-\alpha/2$ term in Algorithm \ref{Conditional Median Interval} is
sharp. However, we can see that if Algorithm \ref{Conditional Median
  Interval} used the $1-\alpha'/2$ quantile of the residuals with
$\alpha'>\alpha$, then by Theorem \ref{Sharpness Case Theorem} there
would exist a choice for $\delta$ and $n$ where the probability of
conditional median coverage would be less than $1-\alpha$. Therefore,
the $1-\alpha/2$ term is required for the probability of coverage to
always be at least $1-\alpha$.

\begin{remark}
  It is possible to generalize Theorem \ref{Sharpness Case Theorem} to
  Algorithm \ref{Conditional Quantile Interval} as well; we can change
  $P^{\delta}$ to have $Y|X \sim X \, B$ with
  $B \sim \textrm{Bernoulli}(q+\textbf{1}[X\geq
  0](1-2q)+\delta)$.
  This results in $\textrm{Quantile}_q(Y|X=x)=x$ for all
  $x\in[-0.5,0.5]$. Then, if we consider the conformity scores
  $f^{\textnormal{lo}}(X_i,Y_i) = f^{\textnormal{hi}}(X_i,Y_i) = Y_i$,
  it can be shown for large $n$ and small $\delta$ that Algorithm
  \ref{Conditional Quantile Interval} returns a confidence interval
  that has a conditional quantile coverage of at most
  $1-\alpha + \epsilon$, meaning that the $rq$ and $1-s(1-q)$ terms in
  the quantiles for the error scores are sharp.
\end{remark}

These results may seem somewhat pedantic because we are restricting
$\hat{\mu}(x)$ to be the zero function and
$f^{\textnormal{lo}}(X_i,Y_i)$ and $f^{\textnormal{hi}}(X_i,Y_i)$ to
be $Y_i$; this simplification is done to better illustrate our
point. Even when $f^{\textnormal{lo}}(X_i,Y_i)$ and
$f^{\textnormal{hi}}(X_i,Y_i)$ are trained using more complicated
approaches, there still exist distributions that result in only
$1-\alpha$ coverage for Algorithm \ref{Conditional Quantile
  Interval}. For an example of a distribution where Algorithm
\ref{Conditional Quantile Interval} only achieves $1-\alpha$ coverage
for standard conformity scores $f^{\textnormal{lo}}(X_i,Y_i)$ and
$f^{\textnormal{hi}}(X_i,Y_i)$, refer to $P_3$ in Section
\ref{Simulations Section}. The existence of $P^{\delta}$ and similarly
`confusing' distributions helps to show why capturing the conditional
median can be tricky in a distribution-free setting.
  
At the same time, there exist conformity scores for which Algorithms
\ref{Conditional Median Interval} and \ref{Conditional Quantile
  Interval} have rates of coverage that are always near
$1-\alpha/2$. For $c > 0$, define the randomized regression function
$\hat{\mu}_c$ as follows: set
$M = \displaystyle \max_{i\in\mathcal{I}_1} |Y_i|$ and
$\hat{\mu}_c(x) = A_x$ for all $x\in\mathbb{R}^d$, where
$A_x\overset{i.i.d.}{\sim} \mathcal{N}(0,(cM)^2)$. We prove the
following:

\begin{theorem}
\label{Crazy Regression Coverage Theorem}
For all $\epsilon > 0$, there exists $c$ and $N$ such that for
  all $n>N$, there is a split $n_1+n_2=n$ such that when Algorithm
\ref{Conditional Median Interval} is ran using the regression function
$\hat{\mu}_c$ on $n$ datapoints with $\mathcal{I}_1$ of size $n_1$ and
$\mathcal{I}_2$ of size $n_2$, the resulting interval will be finite
and will satisfy
$$\mathbb{P}\big\{\textnormal{Median}(Y_{n+1}|X_{n+1})\in\hat{C}_n(X_{n+1}) \big\}\geq 1-\alpha/2 - \epsilon$$
for \textbf{any} distribution $P$.
\end{theorem}

The proof for this theorem is in Appendix \ref{Crazy Regression Coverage Proof}.

\begin{remark}
  Theorem \ref{Crazy Regression Coverage Theorem} can be extended
  to Algorithm \ref{Conditional Quantile Interval} by taking
  $f^{\textnormal{lo}}(X_i,Y_i) = f^{\textnormal{hi}}(X_i,Y_i) = Y_i -
  \hat{\mu}_c(X_i)$. The corresponding result shows that given a large enough number of datapoints and a particular data split, there
  exist conformity scores that result in Algorithm \ref{Conditional
    Quantile Interval} capturing the conditional quantile nontrivially with
  probability at least $1-rq - s(1-q)$ for all distributions $P$.
\end{remark}

Due to the definition of $\hat{\mu}_c$, the resulting confidence
intervals will be near-useless; the predictions will be so far off
that the intervals will have width several times the range of the
(slightly clipped) marginal distribution $P_Y$. However, they still
will be finite, and will still achieve predictive inference at a rate
roughly equal to $1-\alpha/2$.  The existence of scores that always
result in higher-than-needed rates of coverage means that any result
like Theorem \ref{Sharpness Case Theorem} that provides a
  nontrivial upper bound for either Algorithm \ref{Conditional Median
    Interval} or Algorithm \ref{Conditional Quantile Interval}'s
  coverage of the conditional median on a specific distribution will
have to restrict the class of {regression functions and/or conformity scores.}

\section{Median Intervals and Predictive Intervals are Equivalent}
\label{Median Intervals are Predictive Section}

Up until this point, we have looked at the existence and accuracy of
algorithms for estimating the conditional median. 
This section shows that any algorithm for the conditional median is also a predictive algorithm, and vice versa. As a consequence, this means there exists a strong lower
bound on the size of any conditional median confidence interval.

\begin{theorem}
\label{All Conditional Median Algorithms}

Let $\hat{C}_n$ be any algorithm that satisfies $(1-\alpha)$-Median. Then, for any nonatomic distribution $P$ on $\mathbb{R}^d\times \mathbb{R}$, we have that
$$\mathbb{P}\big\{Y_{n+1}\in\hat{C}_n(X_{n+1}) \big\}\geq 1-\alpha.$$
That is, $\hat{C}_n$ satisfies $(1-\alpha)$-Predictive for all nonatomic distributions $P$.
\end{theorem}

\begin{proof}
  The proof above uses the same approach from the proof of Theorem 1
  in \cite{barber2020distribution}. Consider an arbitrary $\hat{C}_n$
  that satisfies $(1-\alpha)$-Median, and let $P$ be any distribution
  over $\mathbb{R}^d\times\mathbb{R}$ for which $P_X$ is
  nonatomic. Pick some $M\geq n+1$, and sample
  $\mathcal{L} = \{(X^j,Y^j):1\leq j\leq
  M\}\stackrel{\text{i.i.d.}}{\sim} P$.
  We define two different ways of sampling our data from
  $\mathcal{L}$.

Fix $\mathcal{L}$ and pick $(X_1,Y_1),\ldots,(X_{n+1},Y_{n+1})$ \textbf{without} replacement from $\mathcal{L}$. Call this method of sampling $Q_1$. It is clear that after marginalizing over $\mathcal{L}$, the $(X_i,Y_i)$'s are effectively drawn i.i.d.~from $P$; thus, we have that
$$\mathbb{P}_{P}\{Y_{n+1}\in \hat{C}_n(X_{n+1})\}=\mathbb{E}_{\mathcal{L}}\big[\mathbb{P}_{Q_1}\{Y_{n+1}\in \hat{C}_n(X_{n+1})|\mathcal{L}\}\big].$$

Now, pick $(X_1,Y_1),\ldots,(X_{n+1},Y_{n+1})$ \textbf{with} replacement from $\mathcal{L}$, and call this method of sampling $Q_2$. Note that because $P_X$ is nonatomic, the $X^j$'s are distinct with probability $1$, which means that $\textrm{Median}_{Q_2}(Y|X=X^j) = Y^j$. Then, as $\hat{C}_n$ applies to all distributions, it applies to our point distribution over $\mathcal{L}$; thus, we have that for all $\mathcal{L}$,
$$\mathbb{P}_{Q_2}\{Y_{n+1}\in \hat{C}_n(X_{n+1})|\mathcal{L}\} = \mathbb{P}_{Q_2}\{\textrm{Median}_{Q_2}(Y|X=X_{n+1})\in \hat{C}_n(X_{n+1})|\mathcal{L}\} \geq 1-\alpha.$$ 

Let $R$ be the event that any two of the $(X_i,Y_i)$ are equal to each other; that is, $R = \{(X_a,Y_a)=(X_b,Y_b)$ $\textrm{ for any }a < b\}$. Note that under $Q_2$, for $1\leq a<b\leq n+1$, the probability of $(X_a,Y_a)$ and $(X_b,Y_b)$ being equal is $1/M$.
Then, by the union bound, 
$$\mathbb{P}_{Q_2}\{R\}\leq \sum_{1\leq a<b\leq n+1}\mathbb{P}_{Q_2}\{(X_a,Y_a)=(X_b,Y_b)\}\leq \frac{n^2}{M},$$
where the last step is from the fact that the number of possible pairs $(a,b)$ is bounded above by $n^2$. Meanwhile, we know that $\mathbb{P}_{Q_1}\{R\}= 0$ by the definition of $Q_1$.

We can use this to bound the total variation distance between $Q_1$ and $Q_2$. For any fixed $\mathcal{L}$ and any event $E$, note that $\mathbb{P}_{Q_1}\{E\} = \mathbb{P}_{Q_1}\{E|R^C\} = \mathbb{P}_{Q_2}\{E|R^C\}$. Then, we can calculate
\begin{align*}
|\mathbb{P}_{Q_1}\{E\} - \mathbb{P}_{Q_2}\{E\}| =& 
|\mathbb{P}_{Q_1}\{E|R^C\} - (\mathbb{P}_{Q_2}\{E|R^C\}\mathbb{P}_{Q_2}\{R^C\} + \mathbb{P}_{Q_2}\{E|R\}\mathbb{P}_{Q_2}\{R\})| \\
=& |\mathbb{P}_{Q_2}\{E|R^C\}\mathbb{P}_{Q_2}\{R\} - \mathbb{P}_{Q_2}\{E|R\}\mathbb{P}_{Q_2}\{R\}| \\
=& \mathbb{P}_{Q_2}\{R\} |\mathbb{P}_{Q_2}\{E|R^C\}- \mathbb{P}_{Q_2}\{E|R\}| \\
\leq & n^2/M.
\end{align*}

Therefore, for any fixed $\mathcal{L}$, the total variation distance between the distributions $Q_1$ and $Q_2$ is at most $n^2/M$, implying that
$$\mathbb{P}_{Q_1}\{Y_{n+1}\in \hat{C}_n(X_{n+1})|\mathcal{L}\}\geq \mathbb{P}_{Q_2}\{Y_{n+1}\in \hat{C}_n(X_{n+1}))|\mathcal{L}\}-n^2/M\geq 1-\alpha-n^2/M,$$
which means that 
$$\mathbb{P}_{P}\{Y_{n+1}\in \hat{C}_n(X_{n+1})\}=\mathbb{E}_{\mathcal{L}}\big[\mathbb{P}_{Q_1}\{Y_{n+1}\in \hat{C}_n(X_{n+1})|\mathcal{L}\}\big]\geq 1-\alpha-n^2/M.$$

Taking the limit as $M$ goes to infinity gives the result.
\end{proof}

\begin{remark}
  Theorem \ref{All Conditional Median Algorithms} also applies to all
  algorithms that satisfy $(1-\alpha,q)$-Quantile; for our uniform
  distribution over $\mathcal{L}$,
  $\textrm{Quantile}_q(Y|X=X^j)=Y^j$, so the proof translates
  exactly. As a result, this means that all algorithms that satisfy
  $(1-\alpha,q)$-Quantile also satisfy $(1-\alpha)$-Predictive for all
  nonatomic distributions $P$.
\end{remark}

\begin{remark}
  The approach taken in the proof of Theorem \ref{All Conditional
    Median Algorithms} is similar to those used to show the limits of
  distribution-free inference in other settings. As mentioned earlier,
  \cite{barber2020distribution} shows that in the setting of
  distributions $P$ over $\mathbb{R}^d\times \{0,1\}$, any confidence
  interval $\hat{C}_n(X_{n+1})$ for $\mathbb{E}[Y|X=X_{n+1}]$ with
  coverage $1-\alpha$ must contain $Y_{n+1}$ with probability
  $1-\alpha$ for all nonatomic distributions $P$, and goes on to
  provide a lower bound for the length of the confidence
  interval. Additionally, \cite{foygel2021limits} proves a similar
  theorem about predictive algorithms $\hat{C}_n(X_{n+1})$ for
  $Y_{n+1}$ that are required to have a weak form of conditional
  coverage. The proof for the result from
  \cite{barber2020distribution} involves the same idea of
  marginalizing over a large finite sampled subset $\mathcal{L}$ in
  order to apply $\hat{C}_n$ to the distribution over $\mathcal{L}$;
  the proof for the result from \cite{foygel2021limits} focuses on
  sampling a large number of datapoints conditioned on whether or not
  they belong to a specific subset
  $\mathcal{B}\subseteq\mathbb{R}^d\times\mathbb{R}$. In both cases,
  studying two sampling distributions and measuring the total
  variation distance between them was crucial. Thus, it seems that
  this strategy may have further use in the future when studying
  confidence intervals for other parameters or data in a
  distribution-free setting.
\end{remark}

We now prove a similar result in the opposite direction.

\begin{theorem}
\label{All Predictive Algorithms}

Let $\hat{C}_n$ be any algorithm that only outputs confidence intervals and satisfies $(1-\alpha/2)$-Predictive. Then, $\hat{C}_n$ satisfies $(1-\alpha)$-Median.
\end{theorem}

\begin{proof}

Consider any distribution $P$. For all $x\in\mathbb{R}^d$ in the
support of $P$, set
$m(x) = \textrm{Median}(Y|X=x)$.

We know that under the distribution $P$, $\mathbb{P}\{Y_{n+1}\in\hat{C}_n(X_{n+1}) \}\geq 1-\alpha/2$. Then, we can condition on whether or not the conditional median is contained in $\hat{C}_n$ as follows:
\begin{align*}
    1-\alpha/2 \leq &\mathbb{P}\{Y_{n+1}\in\hat{C}_n(X_{n+1}) \} \\
    =& \mathbb{P}\{Y_{n+1}\in\hat{C}_n(X_{n+1}) | m(X_{n+1})\in\hat{C}_n(X_{n+1})\}\mathbb{P}\{ m(X_{n+1})\in\hat{C}_n(X_{n+1})\}\\
    &+ \mathbb{P}\{Y_{n+1}\in\hat{C}_n(X_{n+1}) | m(X_{n+1})\not\in\hat{C}_n(X_{n+1})\}\mathbb{P}\{ m(X_{n+1})\not\in\hat{C}_n(X_{n+1})\} \\
    \leq& \mathbb{P}\{ m(X_{n+1})\in\hat{C}_n(X_{n+1})\}+ \frac{1}{2}\mathbb{P}\{ m(X_{n+1})\not\in\hat{C}_n(X_{n+1})\} \\
    =& \frac{1}{2} + \frac{1}{2}\mathbb{P}\{ m(X_{n+1})\in\hat{C}_n(X_{n+1})\}.
\end{align*}
The key insight here is that because $\hat{C}_n$ only outputs confidence intervals, $\hat{C}_n(X_{n+1})$ can cover at most half of the conditional distribution of $Y|X = X_{n+1}$ if $m(X_{n+1})$ is not contained in the confidence interval.

Shifting the constant over and multiplying by $2$ tells us that $\mathbb{P}\{ m(X_{n+1})\in\hat{C}_n(X_{n+1})\}\geq 1-\alpha$, as desired.

\end{proof}

\begin{remark}
\label{All Predictive Algorithms Remark}
Theorem \ref{All Predictive Algorithms} can be modified to replace $(1-\alpha)$-Median with $(1-\alpha,q)$-Quantile. In particular, let $\hat{C}_n(x) = [\hat{L}_n(x), \hat{H}_n(x)]$ be any algorithm that only outputs confidence intervals and satisfies $\mathbb{P}\{Y_{n+1} \geq \hat{L}_n(X_{n+1})\}\geq 1-rq$ and $\mathbb{P}\{Y_{n+1} \leq \hat{H}_n(X_{n+1})\}\geq 1-s(1-q)$ for some $r+s=\alpha$. Then, $\hat{C}_n$ satisfies $(1-\alpha,q)$-Quantile. 

Alternatively, we can also conclude that if $\hat{C}_n$ only outputs confidence intervals and satisfies $(1-\min\{q,1-q\}\alpha)$-Predictive, then it satisfies $(1-\alpha,q)$-Quantile. The proofs of both of these statements are in Appendix \ref{All Predictive Intervals Extension Proof}.

\end{remark}

Theorem \ref{All Conditional Median Algorithms} tells us that
conditional median inference is at least as imprecise as predictive
inference. As a result, because all predictive intervals have
nonvanishing widths (assuming nonzero conditional variance) no matter
the sample size $n$, it is not possible to write down conditional
median algorithms with widths converging to $0$. Thus, it may be
better to study other distribution parameters similar to the conditional median if we are looking
for better empirical performance. For discussion on related distribution parameters that are worth studying and may result in stronger inference, refer to Section \ref{Further Work Section}. 

Theorem \ref{All Predictive Algorithms} tells us that one way to approach conditional median inference is to apply strong predictive algorithms. It also suggests that improvements in predictive inference may translate to conditional median inference. 

Lastly, we know that Algorithm \ref{Conditional Median Interval}
captures $Y_{n+1}$ with probability $1-\alpha/2$. Because there is space between the bounds of Theorems \ref{All Conditional Median Algorithms} and \ref{All Predictive Algorithms}, there may
exist a better conditional median algorithm that only captures
$Y_{n+1}$ with probability $1-\alpha$. Based on our result from
Section \ref{Algorithm Sharpness Section}, any such algorithm will
likely follow a format different than the split conformal
approach. Studying this problem in more detail, particularly on
difficult distributions $P$, might lead to more accurate conditional
median algorithms.

\section{Simulations}
\label{Simulations Section}

In this section, we analyze the impact of different conformity scores
on the outcome of Algorithm \ref{Conditional Quantile
  Interval}. Specifically, we look at the four following conformity
scores:
\begin{enumerate}[align=left, itemsep =-2pt, topsep = 5pt]
    \item [Score 1:] $f_1^{\textnormal{lo}}(X_i,Y_i) = f_1^{\textnormal{hi}}(X_i,Y_i) = Y_i-\hat{\mu}(X_i)$. We train $\hat{\mu}$ to predict the conditional mean using quantile regression forests on the dataset $\{(X_i,Y_i):i\in\mathcal{I}_1\}$. 
    \item [Score 2:] $f_2^{\textnormal{lo}}(X_i,Y_i) = f_2^{\textnormal{hi}}(X_i,Y_i)= \frac{Y_i-\hat{\mu}(X_i)}{\hat{\sigma}(X_i)}$. We train $\hat{\mu}$ and $\hat{\sigma}$ jointly using random forests on the dataset $\{(X_i,Y_i):i\in\mathcal{I}_1\}$.
    \item [Score 3:]  $f_3^{\textnormal{lo}}(X_i,Y_i) = Y_i - \hat{Q}^{\textnormal{lo}}(X_i)$ and $f_3^{\textnormal{hi}}(X_i,Y_i) = Y_i - \hat{Q}^{\textnormal{hi}}(X_i)$. We train $\hat{Q}^{\textnormal{lo}}$ and $\hat{Q}^{\textnormal{hi}}$ to predict the conditional $\alpha/2$ quantile and $1-\alpha/2$ quantile, respectively, using quantile regression forests on the dataset $\{(X_i,Y_i):i\in\mathcal{I}_1\}$. 
    \item [Score 4:]  $f_4^{\textnormal{lo}}(X_i,Y_i)
      =f_4^{\textnormal{hi}}(X_i,Y_i) = \hat{F}_{Y|X=X_i}(Y_i)$. We
      create $\hat{F}_{Y|X=X_i}(Y_i)$ by using $101$ quantile
      regression forests trained on $\{(X_i,Y_i):i\in\mathcal{I}_1\}$
      to estimate the conditional $q^{\textnormal{th}}$ quantile for
      $q\in\{0,0.01,\ldots,0.99,1\}$ and use linear interpolation
      between quantiles to estimate the conditional CDF. Our training method ensures that quantile predictions will never cross.
\end{enumerate}

To compare the performance of Algorithm \ref{Conditional Quantile Interval} against a method that does not have a distribution-free guarantee, we also consider a nonconformalized algorithm that outputs $\hat{C}_n(X_{n+1}) = [\hat{Q}^{\textnormal{lo}}(X_{n+1}), \hat{Q}^{\textnormal{hi}}(X_{n+1})]$, where $\hat{Q}^{\textnormal{lo}}$ and $\hat{Q}^{\textnormal{hi}}$ are trained on $\mathcal{D}$ to predict the conditional $\alpha/2$ quantile and $1-\alpha/2$ quantile, respectively, using quantile regression forests. We refer to this algorithm as QRF.

In order to test the conditional median coverage rate, we must look at
distributions for which the conditional median is known and,
therefore, focus on simulated datasets. We consider the performance of
Algorithm \ref{Conditional Quantile Interval} and QRF on these three
distributions:
\begin{enumerate}[align=left, itemsep =-2pt, topsep = 5pt]
    \item [Distribution 1: ] We draw $(X,Y)\sim P_1$ from $\mathbb{R}^{d}\times\mathbb{R}$, where $d=10$. Here, $X=(X^1,\ldots,X^d)$ is an equicorrelated multivariate Gaussian vector with mean zero and $\textrm{Var}(X^i)=1$, $\textnormal{Cov}(X^i,X^j)=0.25$ for $i\neq j$. We set $Y = (X^1+X^2)^2-X^3+\sigma(X)\epsilon$, where $\epsilon\sim \mathcal{N}(0,1)$ is independent of $X$ and $\sigma(x) = 0.1 + 0.25\|x\|_2^2$ for all $x\in\mathbb{R}^d$.
    \item [Distribution 2: ] We draw $(X,Y)\sim P_2$ from $\mathbb{R}\times\mathbb{R}$. Draw $X\sim\textrm{Unif}[-4\pi,4\pi]$ and $Y=U^{1/4}f(X)$, where $U\sim\textrm{Unif}[0,1]$ is independent of $X$ and $f(x) = 1+|x|\sin^2(x)$ for all $x\in\mathbb{R}$.
    \item [Distribution 3: ] We draw $(X,Y)\sim P_3$ from $\mathbb{R}\times\mathbb{R}$. Draw $X\sim\textrm{Unif}[-1,1]$ and set $Y=B\cdot f(X)$, where $B\sim\textrm{Bernoulli}(0.5+2\delta)$ is independent of $X$ and $f(x)=\gamma\{Mx\}-\frac{\gamma}{2}-(-1)^{\lfloor Mx\rfloor}(1-\frac{\gamma}{2})$ for all $x\in\mathbb{R}$. Note that $\{r\}$ is the fractional part of $r$. We set $\delta = 0.0001$ and $M = 1/\gamma = 25$.
\end{enumerate}
Distributions 2 and 3 are shown in Figure \ref{Figure: Examples of
  Distributions}.

\begin{figure}[ht]
\centering
\begin{subfigure}{0.48\linewidth}
\centering
\captionsetup{justification=centering,margin=0.25cm}
\includegraphics[width=\linewidth]{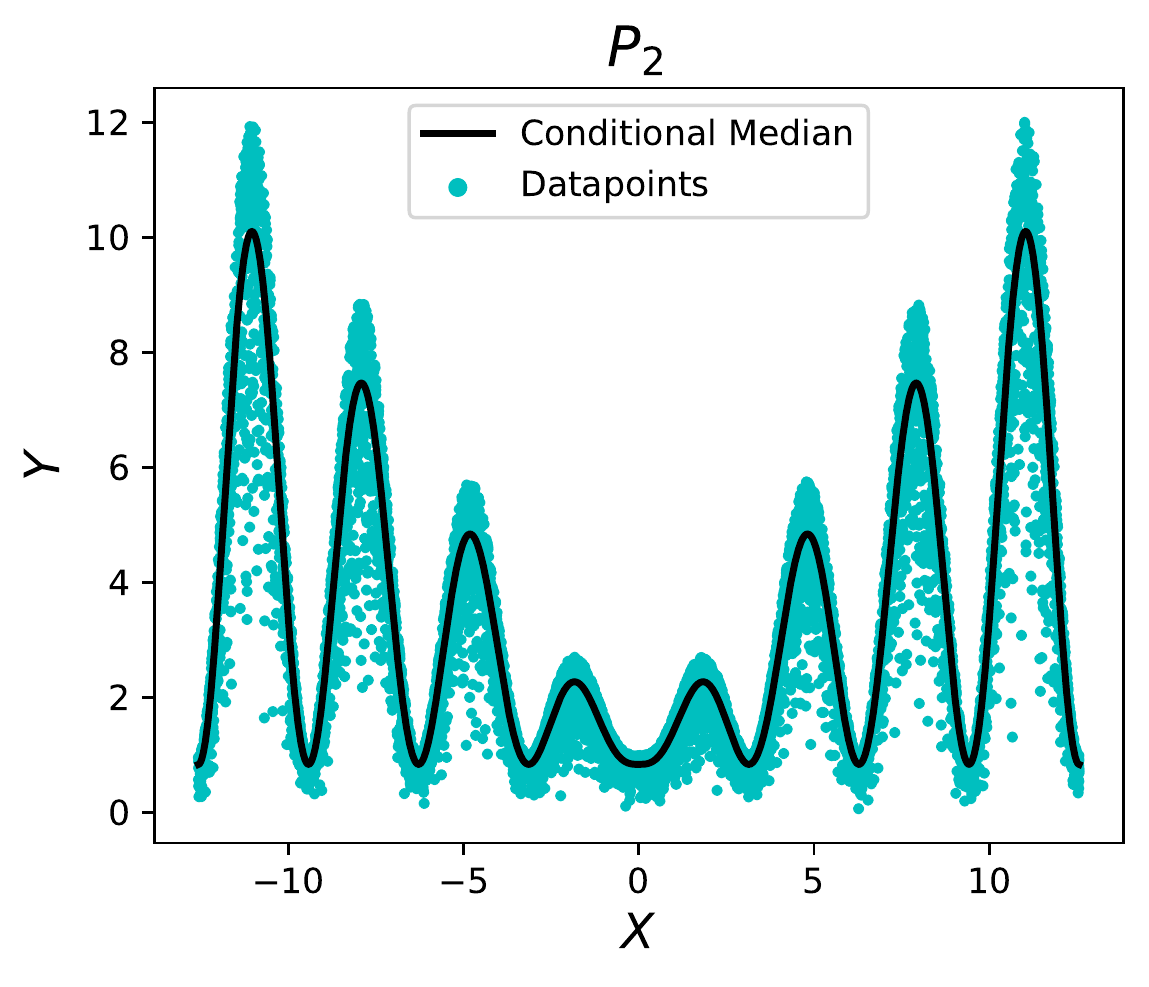}
\end{subfigure}
\begin{subfigure}{0.48\linewidth}
\centering
\captionsetup{justification=centering,margin=0.25cm}
\includegraphics[width=\linewidth]{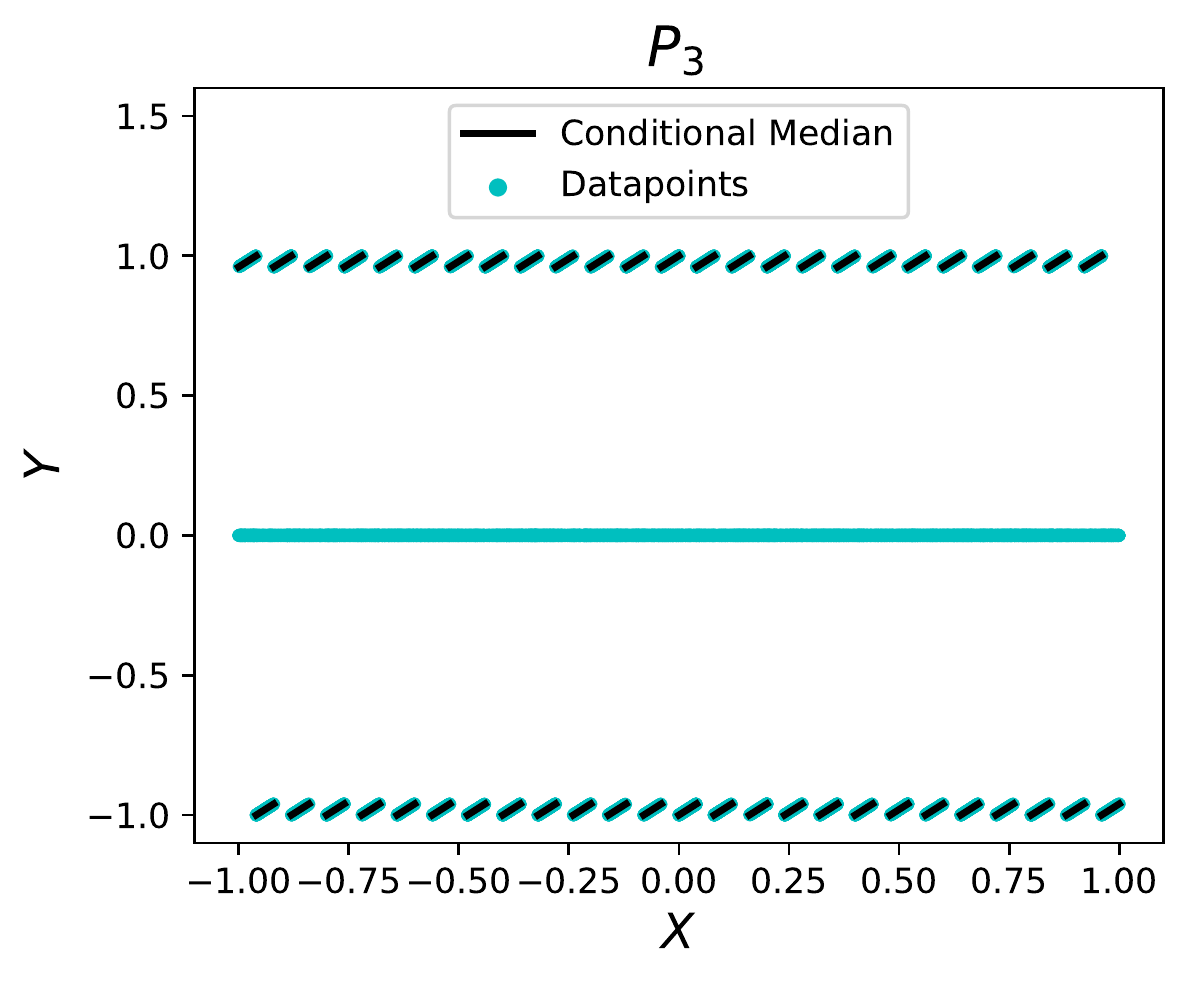}
\end{subfigure}
\caption{Plots of $n=10,000$ datapoints from the two distributions
  $P_2$ and $P_3$ overlaid with the conditional medians.  The left
  panel is a case where the conditional distribution $Y|X$ has high
  heteroskedasticity. The right is a case where it is nearly
  impossible to tell the location of the conditional median.}
\label{Figure: Examples of Distributions}
\end{figure}

For each distribution, we run Algorithm \ref{Conditional Quantile Interval} using each conformity score, as well as QRF, to get a confidence interval for the conditional median. We run 500 trials; in each trial, we set $n=5,000$ with $n_1=n_2=n/2$ and $\alpha=0.1$ with $r=s=\alpha/2$. For a separate study on the impact of $n$ on the confidence interval width and coverage, we refer the reader to \cite{sesia2020comparison}. We test coverage on $5,000$ datapoints for each trial. The average coverage rate, average interval width, and other statistics for each distribution and conformity score are shown in Figure \ref{Figure: Coverage Rates and Interval Widths}. An example of the resulting confidence intervals for a single trial on Distribution 2 are displayed in Figure \ref{Figure: Confidence Interval Examples}.

\begin{figure}[H]
    \centering
    \def\arraystretch{1.25}
    \begin{tabular}{|c|c|c|c|c|c|c|}
    \hline
    Distribution & Score & AC & SDAC & MCC & AW & SDAW \\
    \Xhline{4\arrayrulewidth}
    \multirow{5}{*}{1} & 1 & $99.08\%$ & $0.20\%$ & $0.0\%$ & $14.08$ & $0.573$ \\
    \cline{2-7}
     & 2 & $99.78\%$ & $0.09\%$ & $5.0\%$ & $13.03$ & $0.405$ \\
    \cline{2-7}
     & 3 & $99.72\%$ & $0.10\%$ & $8.4\%$ & $12.86$ & $0.358$ \\
    \cline{2-7}
     & 4 & $99.77\%$ & $0.09\%$ & $12.0\%$ & $13.02$ & $0.398$ \\
    \cline{2-7}
     & QRF & $99.73\%$ & $0.09\%$ & $1.6\%$ & $11.10$ & $0.191$ \\
    \Xhline{4\arrayrulewidth}
    \multirow{5}{*}{2} & 1 & $99.85\%$ & $0.20\%$ & $92.4\%$ & $4.537$ & $0.174$ \\
    \cline{2-7}
     & 2 & $99.48\%$ & $0.29\%$ & $78.4\%$ & $3.604$ & $0.093$ \\
    \cline{2-7}
     & 3 & $99.89\%$ & $0.14\%$ & $95.6\%$ & $3.619$ & $0.060$ \\
    \cline{2-7}
     & 4 & $99.87\%$ & $0.14\%$ & $93.6\%$ & $3.700$ & $0.087$ \\
    \cline{2-7}
     & QRF & $99.91\%$ & $0.12\%$ & $93.4\%$ & $3.48$ & $0.051$ \\
    \Xhline{4\arrayrulewidth}
    \multirow{5}{*}{3} & 1 & $90.05\%$ & $0.86\%$ & $15.4\%$ & $2.122$ & $0.044$ \\
    \cline{2-7}
     & 2 & $89.97\%$ & $0.87\%$ & $19.8\%$ & $2.084$ & $0.051$ \\
    \cline{2-7}
     & 3 & $90.14\%$ & $0.87\%$ & $4.4\%$ & $1.989$ & $0.003$ \\
    \cline{2-7}
     & 4 & $90.00\%$ & $0.88\%$ & $0.0\%$ & $1.990$ & $0.002$ \\
    \cline{2-7}
     & QRF & $83.98\%$ & $0.92\%$ & $0.0\%$ & $1.962$ & $0.017$ \\
    \Xhline{4\arrayrulewidth}
    \end{tabular}
    
    \caption{For each distribution and conformity score, we calculate: average coverage (AC), an estimate of $\mathbb{P}\{\textrm{Median}(Y_{n+1}|X_{n+1})\in \hat{C}_n(X_{n+1})\}$; standard deviation of average coverage (SDAC), an estimation of $\textrm{Var}(\mathbb{P}\{\textrm{Median}(Y_{n+1}|X_{n+1})\in \hat{C}_n(X_{n+1})|\mathcal{D}\})^{1/2}$, where $\mathcal{D} = \{(X_1,Y_1),\ldots,(X_n,Y_n)\}$; minimum conditional coverage (MCC), an estimate of $\displaystyle\min_x
        \mathbb{P}\{\textrm{Median}(Y|X=x)\in\hat{C}_n(x)\}$; average width (AW), an estimate of $\mathbb{E}[\textrm{len}(\hat{C}_n(X_{n+1}))]$; and standard deviation of average width (SDAW), an estimate of $\textrm{Var}(\mathbb{E}[\textrm{len}(\hat{C}_n(X_{n+1}))|\mathcal{D}])^{1/2}$. Estimations are averaged over 500 trials. $1-\alpha=0.9$ for all trials.}
    \label{Figure: Coverage Rates and Interval Widths}
\end{figure}

\begin{figure}[!ht]
\centering
\begin{subfigure}{0.45\linewidth}
\centering
\captionsetup{justification=centering,margin=0.25cm}
\caption{$f_1^{\textnormal{lo}}(X_i,Y_i) =
  f_1^{\textnormal{hi}}(X_i,Y_i) = Y_i-\hat{\mu}(X_i)$. $99.50\%$
  coverage; $4.52$ average width.} 
\includegraphics[width=\linewidth]{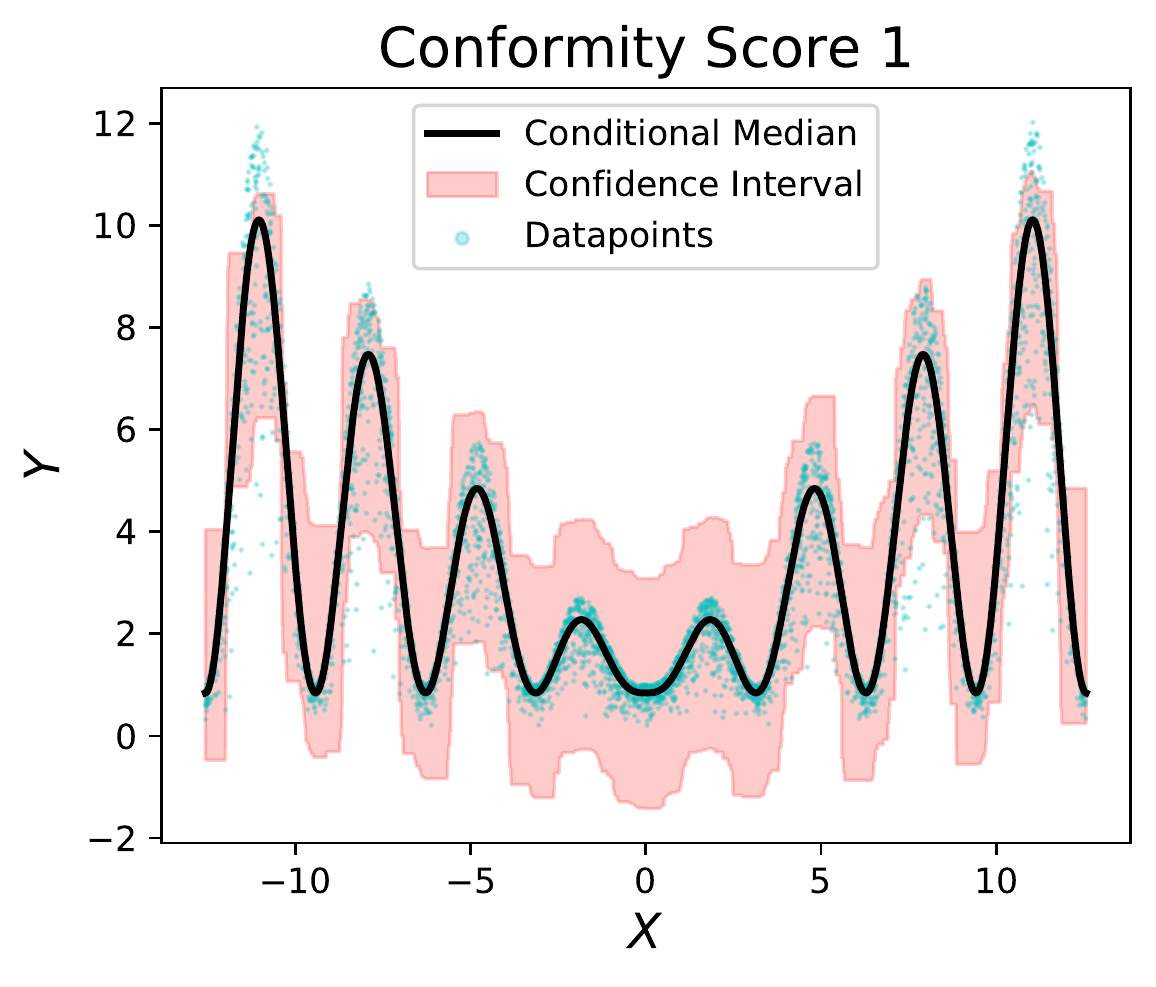}
\end{subfigure}
\begin{subfigure}{0.45\linewidth}
\centering
\captionsetup{justification=centering,margin=0.5cm}
\caption{$f_2^{\textnormal{lo}}(X_i,Y_i) = f_2^{\textnormal{hi}}(X_i,Y_i)= \frac{Y_i-\hat{\mu}(X_i)}{\hat{\sigma}(X_i)}$. $99.76\%$ coverage; $3.61$ average width.}
\includegraphics[width=\linewidth]{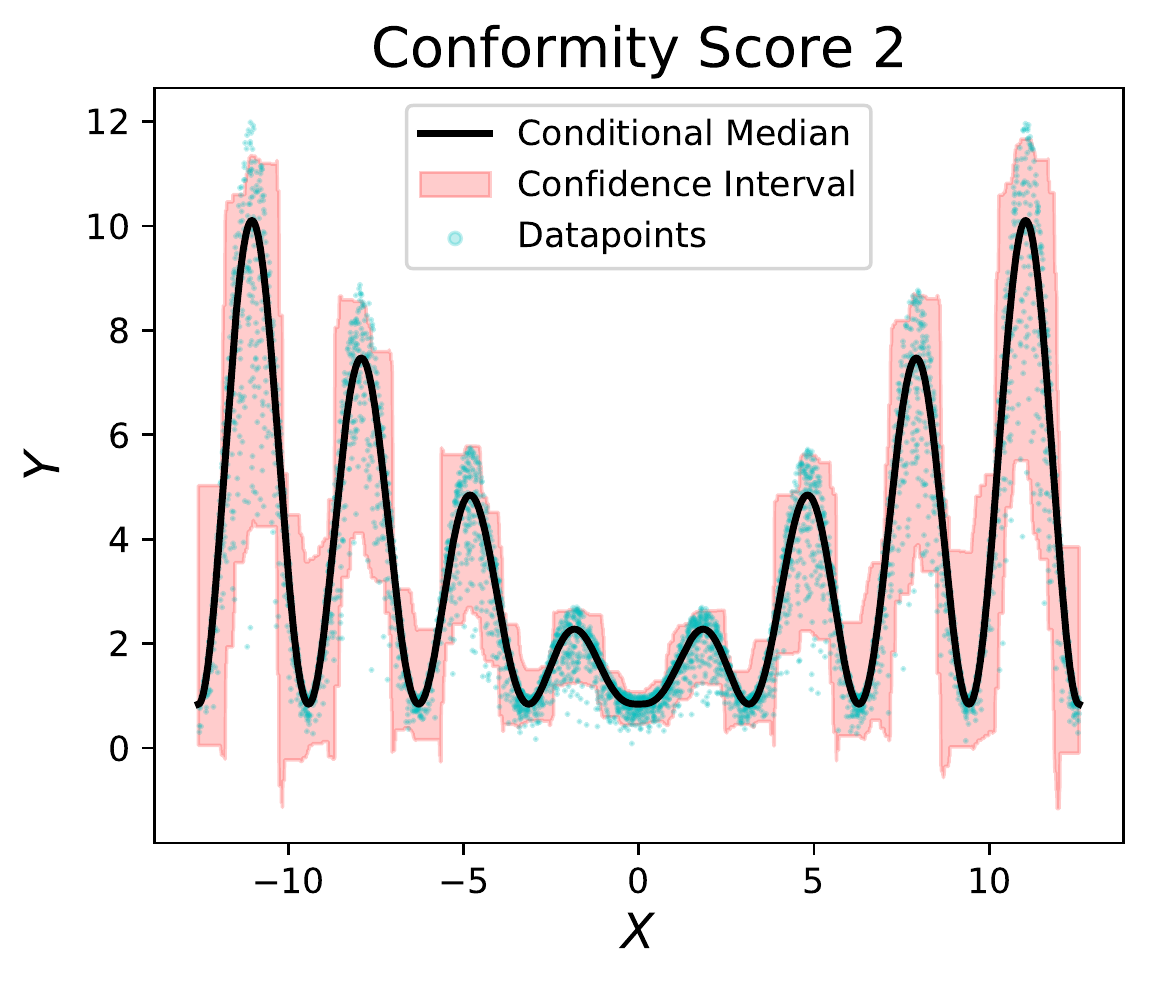}
\end{subfigure}
\begin{subfigure}{0.45\linewidth}
\centering
\captionsetup{justification=centering,margin=0.25cm}
\caption{$f_3^{\textnormal{lo}}(X_i,Y_i) = Y_i - \hat{Q}^{\textnormal{lo}}(X_i)$; $f_3^{\textnormal{hi}}(X_i,Y_i) = Y_i - \hat{Q}^{\textnormal{hi}}(X_i)$. $99.72\%$ coverage; $3.61$ average width.}
\includegraphics[width=\linewidth]{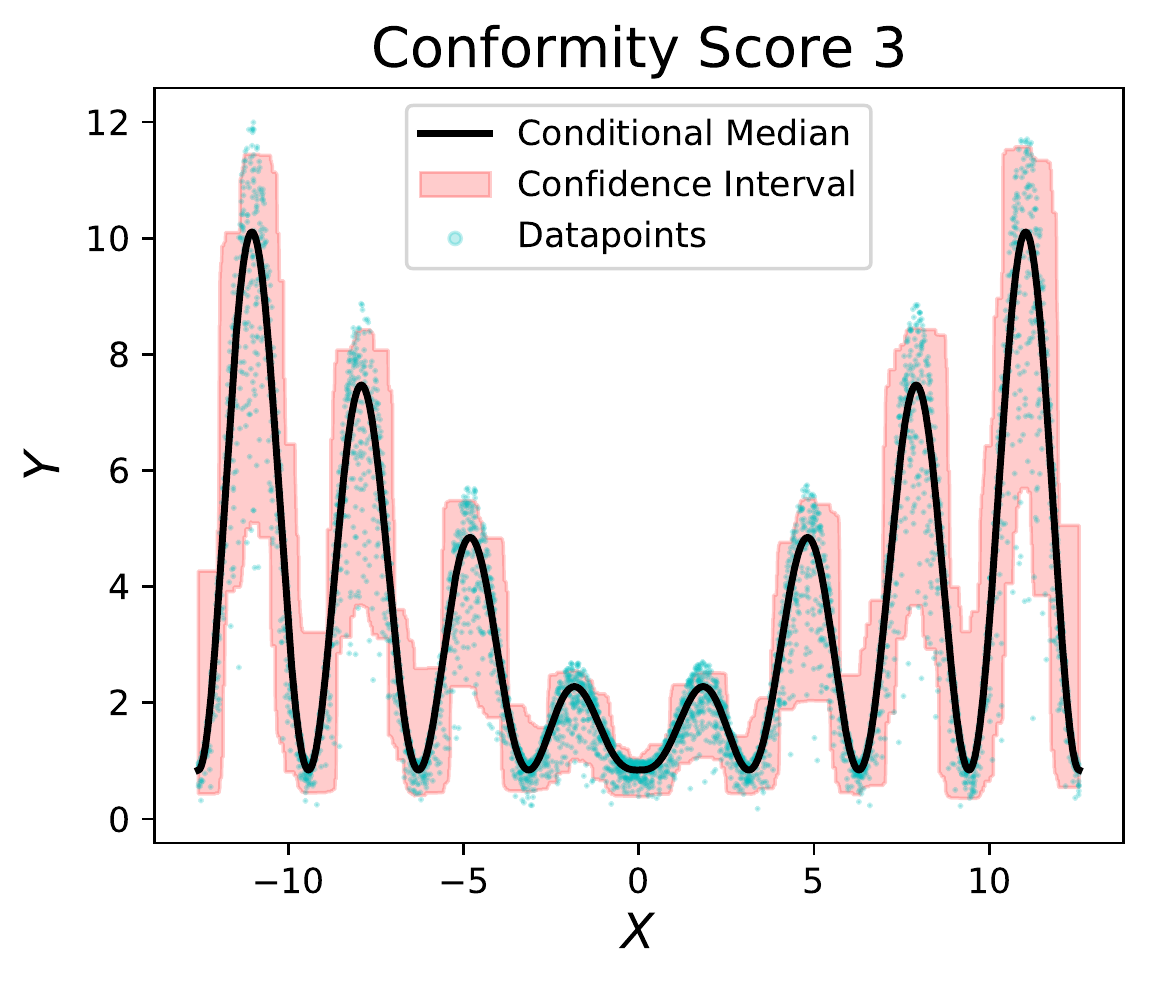}
\end{subfigure}
\begin{subfigure}{0.45\linewidth}
\centering
\captionsetup{justification=centering,margin=0.25cm}
\caption{ $f_4^{\textnormal{lo}}(X_i,Y_i) =f_4^{\textnormal{hi}}(X_i,Y_i) = \hat{F}_{Y|X=X_i}(Y_i)$. $100.0\%$ coverage; $3.66$ average width.}
\includegraphics[width=\linewidth]{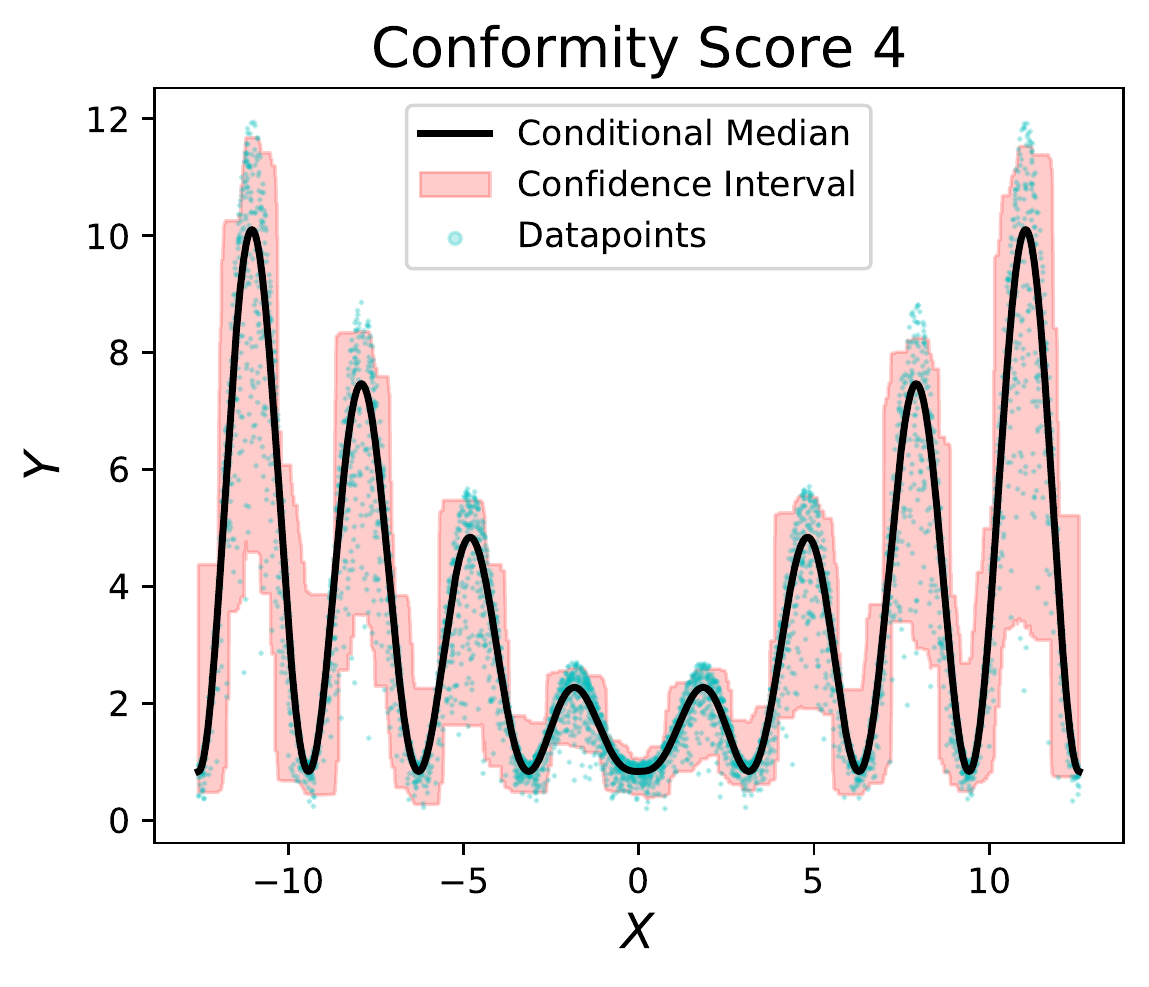}
\end{subfigure}
\caption{Confidence intervals (pink regions) from one trial for each conformity score
  on Distribution 2. Note that all scores result in a coverage well
  over $1-\alpha = 0.9$.}
\label{Figure: Confidence Interval Examples}
\end{figure}

Comparing the QRF algorithm with Algorithm \ref{Conditional Quantile Interval}, we see that while the widths are significantly lower for all three distributions, the coverage for Distribution 3 is much less than $1-\alpha$. In particular, compared against Conformity Score 3, which has the same framework but includes a constant buffer on both ends due to the calibration set, we see how much extra width Algorithm \ref{Conditional Quantile Interval} adds in the calibration step for Conformity Score 3.

Looking first at rates of coverage, we see that all scores have coverage much greater than $1-\alpha$ on Distributions 1 and 2. However, Distribution 3 is a case where all scores have near-identical rates of coverage at just about $1-\alpha$. Further investigation into the confidence intervals produced for Distribution 3 suggests that the algorithms are often failing to capture the conditional median when its absolute value is almost exactly $1$.

The minimum conditional coverage on Distributions 1 and 3 is near 0 for each conformity score. Interestingly, the scores with the worst minimum conditional coverage on Distribution 1 have the relative best minimum conditional coverage on Distribution 3, and vice versa. Scores 1, 3, and 4 have a minimum conditional coverage greater than $1-\alpha$ on Distribution 2, implying that these scores achieve point-wise conditional coverage.

Regarding interval width, Score 1 performs significantly worse on Distributions 1 and 2 than all other scores; meanwhile, the 3 other scores have roughly equal average widths. On Distribution 3, Scores 3 and 4 produce intervals with significantly less width than Scores 1 and 2.

Overall, we see that Score 1 is significantly worse than the other
scores on distributions with a wide range in conditional variance;
Scores 3 and 4 behave very similarly on all distributions and perform
slightly better than Score 2 on Distribution 3.

\pagebreak
\section{Discussion}

This paper introduced two algorithms for capturing the conditional median
of a datapoint within the distribution-free setting, as well as a particular distribution where
the performance of these algorithms was sharp. Our lower bounds prove that in the
distribution-free setting, conditional median inference is
fundamentally as difficult as prediction itself, thereby setting a
concrete limit to how well any median inference algorithm can ever perform. We also showed that any predictive algorithm can be used as a median algorithm at a different coverage level, suggesting that the two problems are near-equivalent.

\subsection{Takeaways}

A few observations may prove useful. For one, distributions such as
$P^{\delta}$ from Section \ref{Algorithm Sharpness Section} and $P_3$
from Section \ref{Simulations Section} will likely show up
again. Because each distribution is a mixture of two disjoint
distributions with roughly equal weights, it is hard to identify which
half contains the median. It is likely that similar distributions will
show up as the performance-limiting distribution for distribution-free
parameter inference. Further, the proof technique of sampling a large
finite number of datapoints and then marginalizing (Section
\ref{Median Intervals are Predictive Section}) is similar to those in
\cite{barber2020distribution} and \cite{foygel2021limits}, pointing
out to possible future use. Lastly, our results and those of
\cite{barber2020distribution} indicate that the value of conditional
parameters cannot be known with higher accuracy than the values of
future samples.

\subsection{Further Work}
\label{Further Work Section}

We hope that this paper motivates further work on conditional
parameter inference. We see three immediate potential avenues:
\begin{itemize}
\item One direction is to extend our methods to study other
  conditional parameters similar to the conditional median. For example, the smoothed conditional median, equal to the conditional median convolved with a kernel, may be easier to infer than the conditional median. This parameter would allow for smarter inference in the case of smooth distributions without making smoothness a requirement for inference. Similarly, the truncated mean and other measures of central tendency may be amenable to model-free inference and analysis, as may the conditional interquartile range and other robust measures of scale. 

\item Another direction is to get tighter bands by imposing mild shape
  constraints on the conditional median function. For instance, if we
  know that $\textrm{Median}(Y|X=x)$ is convex, then the
  results from Section \ref{Median Intervals are Predictive Section}
  no longer apply. Similarly, assuming that $\textrm{Median}(Y|X=x)$ is decreasing in
  $x$ or Lipschitz would yield intervals with vanishing widths in the
  limit of large samples. For instance, when predicting economic
  damages caused by tornadoes using wind speed as a covariate, one may
  assume that the median damage is nondecreasing as wind speed
  increases. 

\item A third subject of study is creating \textit{full conformal
    inference} methods based off of our split conformal
  algorithms. Unlike split conformal inference, the full conformal
  method does not rely on splitting the dataset into a fitting half
  and a ranking half; instead, it calculates the conformity of a
  potential datapoint $(X_{n+1},y)$ to the full dataset $\mathcal{D}$
  and includes $y$ in its confidence region only if $(X_{n+1},y)$ is
  similar enough to the observed datapoints. {The study of full
    conformal inference has grown alongside that of split conformal
    inference; the method can be seen in \cite{vovk2005algorithmic},
    \cite{shafer2008tutorial}, and \cite{lei2018distribution}.}
  Standard full conformal algorithms do not guarantee coverage of the
  conditional median; however, there may exist modifications similar
  to locally nondecreasing conformity scores that result in a full
  conformal algorithm that captures the conditional median.
\end{itemize}

\section*{Acknowledgements}

D.~M.~thanks Stanford University for supporting this research as part
of the Masters program in Statistics. 
E.~C.~was supported by Office of Naval Research grant N00014-20-12157,
by the National Science Foundation grant OAC 1934578, and by the Army
Research Office (ARO) under grant W911NF-17-1-0304.  We thank Lihua
Lei for advice on different approaches and recommended resources on
this topic, as well as the Stanford Statistics department for
listening to a preliminary version of this work. We also thank a referee for encouraging us to deepen the connection between predictive and conditional median confidence intervals.

\bibliographystyle{plainnat}
\nocite{*}
\bibliography{references.bib}

\begin{thebibliography}{32}
\providecommand{\natexlab}[1]{#1}
\providecommand{\url}[1]{\texttt{#1}}
\expandafter\ifx\csname urlstyle\endcsname\relax
  \providecommand{\doi}[1]{doi: #1}\else
  \providecommand{\doi}{doi: \begingroup \urlstyle{rm}\Url}\fi

\bibitem[Bahadur and Savage(1956)]{bahadur1956nonexistence}
Raghu~R Bahadur and Leonard~J Savage.
\newblock The nonexistence of certain statistical procedures in nonparametric
  problems.
\newblock \emph{The Annals of Mathematical Statistics}, 27\penalty0
  (4):\penalty0 1115--1122, 1956.

\bibitem[Barber(2020)]{barber2020distribution}
Rina~Foygel Barber.
\newblock Is distribution-free inference possible for binary regression?
\newblock \emph{Electronic Journal of Statistics}, 14\penalty0 (2):\penalty0
  3487--3524, 2020.

\bibitem[Barber et~al.(2021{\natexlab{a}})Barber, Candes, Ramdas, and
  Tibshirani]{barber2021predictive}
Rina~Foygel Barber, Emmanuel~J Candes, Aaditya Ramdas, and Ryan~J Tibshirani.
\newblock Predictive inference with the jackknife+.
\newblock \emph{The Annals of Statistics}, 49\penalty0 (1):\penalty0 486--507,
  2021{\natexlab{a}}.

\bibitem[Barber et~al.(2021{\natexlab{b}})Barber, Candes, Ramdas, and
  Tibshirani]{foygel2021limits}
Rina~Foygel Barber, Emmanuel~J Candes, Aaditya Ramdas, and Ryan~J Tibshirani.
\newblock The limits of distribution-free conditional predictive inference.
\newblock \emph{Information and Inference: A Journal of the IMA}, 10\penalty0
  (2):\penalty0 455--482, 2021{\natexlab{b}}.

\bibitem[Belloni et~al.(2019)Belloni, Chernozhukov, Chetverikov, and
  Fern{\'a}ndez-Val]{belloni2019conditional}
Alexandre Belloni, Victor Chernozhukov, Denis Chetverikov, and Iv{\'a}n
  Fern{\'a}ndez-Val.
\newblock Conditional quantile processes based on series or many regressors.
\newblock \emph{Journal of Econometrics}, 213\penalty0 (1):\penalty0 4--29,
  2019.

\bibitem[Chen et~al.(2018)Chen, Chun, and Barber]{chen2018discretized}
Wenyu Chen, Kelli-Jean Chun, and Rina~Foygel Barber.
\newblock Discretized conformal prediction for efficient distribution-free
  inference.
\newblock \emph{Stat}, 7\penalty0 (1):\penalty0 e173, 2018.

\bibitem[Chernozhukov et~al.(2019)Chernozhukov, W{\"u}thrich, and
  Zhu]{chernozhukov2019distributional}
Victor Chernozhukov, Kaspar W{\"u}thrich, and Yinchu Zhu.
\newblock Distributional conformal prediction.
\newblock \emph{arXiv preprint arXiv:1909.07889}, 2019.

\bibitem[Claeskens et~al.(2003)Claeskens, Van~Keilegom,
  et~al.]{claeskens2003bootstrap}
Gerda Claeskens, Ingrid Van~Keilegom, et~al.
\newblock Bootstrap confidence bands for regression curves and their
  derivatives.
\newblock \emph{The Annals of Statistics}, 31\penalty0 (6):\penalty0
  1852--1884, 2003.

\bibitem[Cort{\'e}s-Ciriano and Bender(2019)]{cortes2019concepts}
Isidro Cort{\'e}s-Ciriano and Andreas Bender.
\newblock Concepts and applications of conformal prediction in computational
  drug discovery.
\newblock \emph{arXiv preprint arXiv:1908.03569}, 2019.

\bibitem[Cribari-Neto and Lima(2009)]{cribari2009heteroskedasticity}
Francisco Cribari-Neto and Maria da Gl{\'o}ria~A Lima.
\newblock Heteroskedasticity-consistent interval estimators.
\newblock \emph{Journal of Statistical Computation and Simulation}, 79\penalty0
  (6):\penalty0 787--803, 2009.

\bibitem[Guan(2019)]{guan2019conformal}
Leying Guan.
\newblock Conformal prediction with localization.
\newblock \emph{arXiv preprint arXiv:1908.08558}, 2019.

\bibitem[Johansson et~al.(2013)Johansson, Bostr{\"o}m, and
  L{\"o}fstr{\"o}m]{johansson2013conformal}
Ulf Johansson, Henrik Bostr{\"o}m, and Tuve L{\"o}fstr{\"o}m.
\newblock Conformal prediction using decision trees.
\newblock In \emph{2013 IEEE 13th international conference on data mining},
  pages 330--339. IEEE, 2013.

\bibitem[Kivaranovic et~al.(2020)Kivaranovic, Johnson, and
  Leeb]{kivaranovic2020adaptive}
Danijel Kivaranovic, Kory~D Johnson, and Hannes Leeb.
\newblock Adaptive, distribution-free prediction intervals for deep networks.
\newblock In \emph{International Conference on Artificial Intelligence and
  Statistics}, pages 4346--4356. PMLR, 2020.

\bibitem[Koenker(2005)]{koenker2005}
Roger Koenker.
\newblock \emph{Quantile Regression}.
\newblock Econometric Society Monographs. Cambridge University Press, 2005.
\newblock \doi{10.1017/CBO9780511754098}.

\bibitem[Lei et~al.(2013)Lei, Robins, and Wasserman]{lei2013distribution}
Jing Lei, James Robins, and Larry Wasserman.
\newblock Distribution-free prediction sets.
\newblock \emph{Journal of the American Statistical Association}, 108\penalty0
  (501):\penalty0 278--287, 2013.

\bibitem[Lei et~al.(2018)Lei, G’Sell, Rinaldo, Tibshirani, and
  Wasserman]{lei2018distribution}
Jing Lei, Max G’Sell, Alessandro Rinaldo, Ryan~J Tibshirani, and Larry
  Wasserman.
\newblock Distribution-free predictive inference for regression.
\newblock \emph{Journal of the American Statistical Association}, 113\penalty0
  (523):\penalty0 1094--1111, 2018.

\bibitem[Lei and Cand{\`e}s(2020)]{lei2020conformal}
Lihua Lei and Emmanuel~J Cand{\`e}s.
\newblock Conformal inference of counterfactuals and individual treatment
  effects.
\newblock \emph{arXiv preprint arXiv:2006.06138}, 2020.

\bibitem[McCarthy(1965)]{mccarthy1965stratified}
Philip~J McCarthy.
\newblock Stratified sampling and distribution-free confidence intervals for a
  median.
\newblock \emph{Journal of the American Statistical Association}, 60\penalty0
  (311):\penalty0 772--783, 1965.

\bibitem[Noether(1972)]{noether1972distribution}
Gottfried~E Noether.
\newblock Distribution-free confidence intervals.
\newblock \emph{The American Statistician}, 26\penalty0 (1):\penalty0 39--41,
  1972.

\bibitem[Oberhofer and Haupt(2016)]{oberhofer2016asymptotic}
Walter Oberhofer and Harry Haupt.
\newblock Asymptotic theory for nonlinear quantile regression under weak
  dependence.
\newblock \emph{Econometric Theory}, 32\penalty0 (3):\penalty0 686--713, 2016.

\bibitem[Papadopoulos et~al.(2002)Papadopoulos, Proedrou, Vovk, and
  Gammerman]{papadopoulos2002inductive}
Harris Papadopoulos, Kostas Proedrou, Volodya Vovk, and Alex Gammerman.
\newblock Inductive confidence machines for regression.
\newblock In \emph{European Conference on Machine Learning}, pages 345--356.
  Springer, 2002.

\bibitem[Papadopoulos et~al.(2011)Papadopoulos, Vovk, and
  Gammerman]{papadopoulos2011regression}
Harris Papadopoulos, Vladimir Vovk, and Alexander Gammerman.
\newblock Regression conformal prediction with nearest neighbours.
\newblock \emph{Journal of Artificial Intelligence Research}, 40:\penalty0
  815--840, 2011.

\bibitem[Romano et~al.(2019)Romano, Patterson, and
  Candes]{romano2019conformalized}
Yaniv Romano, Evan Patterson, and Emmanuel Candes.
\newblock Conformalized quantile regression.
\newblock In \emph{Advances in Neural Information Processing Systems}, pages
  3543--3553, 2019.

\bibitem[Romano et~al.(2020)Romano, Barber, Sabatti, and
  Cand{\`e}s]{Romano2020With}
Yaniv Romano, Rina Barber, Chiara Sabatti, and Emmanuel Cand{\`e}s.
\newblock With malice toward none: Assessing uncertainty via equalized
  coverage.
\newblock \emph{Harvard Data Science Review}, 04 2020.
\newblock \doi{10.1162/99608f92.03f00592}.

\bibitem[Sesia and Cand{\`e}s(2020)]{sesia2020comparison}
Matteo Sesia and Emmanuel~J Cand{\`e}s.
\newblock A comparison of some conformal quantile regression methods.
\newblock \emph{Stat}, 9\penalty0 (1):\penalty0 e261, 2020.

\bibitem[Shafer and Vovk(2008)]{shafer2008tutorial}
Glenn Shafer and Vladimir Vovk.
\newblock A tutorial on conformal prediction.
\newblock \emph{Journal of Machine Learning Research}, 9\penalty0
  (Mar):\penalty0 371--421, 2008.

\bibitem[Takeuchi et~al.(2006)Takeuchi, Le, Sears, and
  Smola]{takeuchi2006nonparametric}
Ichiro Takeuchi, Quoc~V. Le, Timothy~D. Sears, and Alexander~J. Smola.
\newblock Nonparametric quantile estimation.
\newblock \emph{Journal of Machine Learning Research}, 7\penalty0
  (45):\penalty0 1231--1264, 2006.

\bibitem[Tibshirani et~al.(2019)Tibshirani, Barber, Candes, and
  Ramdas]{tibshirani2019conformal}
Ryan~J Tibshirani, Rina~Foygel Barber, Emmanuel Candes, and Aaditya Ramdas.
\newblock Conformal prediction under covariate shift.
\newblock In \emph{Advances in Neural Information Processing Systems}, pages
  2530--2540, 2019.

\bibitem[Vovk(2015)]{vovk2015cross}
Vladimir Vovk.
\newblock Cross-conformal predictors.
\newblock \emph{Annals of Mathematics and Artificial Intelligence}, 74\penalty0
  (1):\penalty0 9--28, 2015.

\bibitem[Vovk et~al.(2005)Vovk, Gammerman, and Shafer]{vovk2005algorithmic}
Vladimir Vovk, Alex Gammerman, and Glenn Shafer.
\newblock \emph{Algorithmic learning in a random world}.
\newblock Springer Science \& Business Media, 2005.

\bibitem[Vovk et~al.(2009)Vovk, Nouretdinov, Gammerman, et~al.]{vovk2009line}
Vladimir Vovk, Ilia Nouretdinov, Alex Gammerman, et~al.
\newblock On-line predictive linear regression.
\newblock \emph{The Annals of Statistics}, 37\penalty0 (3):\penalty0
  1566--1590, 2009.

\bibitem[Vovk et~al.(2018)Vovk, Nouretdinov, Manokhin, and
  Gammerman]{vovk2018cross}
Vladimir Vovk, Ilia Nouretdinov, Valery Manokhin, and Alexander Gammerman.
\newblock Cross-conformal predictive distributions.
\newblock In \emph{Conformal and Probabilistic Prediction and Applications},
  pages 37--51. PMLR, 2018.

\end{thebibliography}

\pagebreak

\appendix
\section{Theorem Proofs}

\subsection{Proof of Theorem \ref{Conditional Median Theorem}}
\label{Conditional Median Proof}

Theorem \ref{Conditional Median Theorem} follows directly from Theorem \ref{All Predictive Algorithms}. In particular, \cite{vovk2005algorithmic} shows that Algorithm \ref{Conditional Median Interval} satisfies $(1-\alpha/2)$-Predictive. Because Algorithm \ref{Conditional Median Interval} always outputs a confidence interval, we have that Algorithm \ref{Conditional Median Interval} satisfies $(1-\alpha)$-Median by Theorem \ref{All Predictive Algorithms}.

\subsection{Proof of Theorem \ref{Conditional Quantile Theorem}}
\label{Conditional Quantile Proof}

We can prove Theorem \ref{Conditional Quantile Theorem} from the extension of Theorem \ref{All Predictive Algorithms} described in Remark \ref{All Predictive Algorithms Remark}. We know that Algorithm \ref{Conditional Quantile Interval} always outputs a confidence interval because it is the intersection of 2 confidence intervals. Define $E_{n+1}^{\textnormal{lo}}=f^{\textnormal{lo}}(X_{n+1},Y_{n+1}))$ and $E_{n+1}^{\textnormal{hi}}=f^{\textnormal{hi}}(X_{n+1},Y_{n+1}))$. We can bound the probability of $Y_{n+1}$ being at least the lower bound by $\mathbb{P}\{E_{n+1}^{\textnormal{lo}} \geq Q_{rq}^{\textrm{lo}}(E)\}$ and bound the probability of $Y_{n+1}$ being at most the upper bound by $\mathbb{P}\{E_{n+1}^{\textnormal{hi}}\leq Q_{1-s(1-q)}^{\textrm{hi}}(E)\}$.

$Q_{rq}^{\textrm{lo}}(E)$ is defined as the $rq(1+1/n_2)-1/n_2$-th quantile of $\{E_i^{\textnormal{lo}}:i\in\mathcal{I}_2\}$, which is equal to the $\lceil n_2(rq(1+1/n_2)-1/n_2) \rceil = \lceil rq(n_2+1)-1 \rceil $ smallest value of $\{E_i^{\textnormal{lo}}:i\in\mathcal{I}_2\}$. Then, because $\{E_{i}^{\textnormal{lo}}:i\in\mathcal{I}_2\}\cup\{E_{n+1}^{\textnormal{lo}}\}$ are exchangeable, as $f^{\textnormal{lo}}$ is only fit on $\{(X_i,Y_i):i\in\mathcal{I}_1\}$, we have that the distribution of $|\{E_i^{\textnormal{lo}} < E_{n+1}^{\textnormal{lo}}:i\in\mathcal{I}_2\}|$ is bounded above by the uniform distribution on $\{0,1,\ldots,n_2\}$. Therefore, \begin{align*}
\mathbb{P}\{E_{n+1}^{\textnormal{lo}} \geq Q_{rq}^{\textrm{lo}}(E)\}\geq&
\sum_{\lceil rq(n_2+1)-1 \rceil+1}^{n_2}\frac{1}{n_2+1}\\
=&\frac{n_2+1-\lceil rq(n_2+1) \rceil}{n_2+1}\\
\geq& \frac{(n_2+1)(1-rq)}{n_2+1}\\
=&1-rq.
\end{align*}

Similarly, $Q_{1-s(1-q)}^{\textrm{hi}}(E)$ is defined as the 
$(1-s(1-q))(1+1/n_2)$-th quantile of
$\{E^{\textnormal{hi}}_i:i\in\mathcal{I}_2\}$, which equals the
$\lceil n_2(1-s(1-q))(1+1/n_2)\rceil = \lceil (1-s(1-q))(n_2+1)\rceil$
smallest value of
$\{E^{\textnormal{hi}}_i:i\in\mathcal{I}_2\}$. Then, because $\{E_{i}^{\textnormal{hi}}:i\in\mathcal{I}_2\}\cup\{E_{n+1}^{\textnormal{hi}}\}$ are exchangeable, we have that the distribution of $|\{E_i^{\textnormal{hi}} \leq E_{n+1}^{\textnormal{hi}}:i\in\mathcal{I}_2\}|$ is bounded below by the uniform distribution on $\{0,1,\ldots,n_2\}$. Therefore,
\begin{align*}
\mathbb{P}\{E_{n+1}^{\textnormal{hi}}\leq Q_{1-s(1-q)}^{\textrm{hi}}(E)\}\geq&
\sum_{0}^{\lceil (1-s(1-q))(n_2+1)\rceil - 1}\frac{1}{n_2+1}\\
=& \frac{\lceil (1-s(1-q))(n_2+1)\rceil}{n_2+1}\\
\geq&1-s(1-q).
\end{align*}

Combining these two results and applying the extension of Theorem \ref{All Predictive Algorithms} tells us that Algorithm \ref{Conditional Quantile Interval} satisfies $(1-\alpha,q)$-Quantile as desired.

\subsection{Proof of Theorem \ref{Sharpness Case Theorem}}
\label{Sharpness Case Proof}

We show that given $\epsilon$, there exists $\delta$ and $N$ such that
for all $n>N$, running Algorithm \ref{Conditional Median Interval} on
$P^{\delta}$ with our chosen $\hat{\mu}$ results in a confidence
interval that contains the conditional median with probability at most
$1-\alpha+\epsilon$. Our approch is similar to that in Appendix
\ref{Conditional Median Proof}; however, we apply the inequalities in
the opposite directions and use some analysis in order to get an upper
bound as opposed to a lower bound.

First, note that $\{(X_i,Y_i):i\in\mathcal{I}_1\}$ is irrelevant to
our algorithm, as $\hat{\mu}$ is set to be the zero function. Then,
for each $i\in\mathcal{I}_2$, $E_i = |Y_i|$. Thus, $Q_{1-\alpha/2}(E)$
is the $(1-\alpha/2)(1+1/n_2)$-th empirical quantile of
$\{|Y_i|:i\in\mathcal{I}_2\}$, and our confidence interval is
$\hat{C}_n(X_{n+1}) = [-Q_{1-\alpha/2}(E),Q_{1-\alpha/2}(E)]$.
Because the parameter we want to cover is
$\textrm{Median}(Y|X=X_{n+1})=X_{n+1}$, 
$$\textrm{Median}(Y_{n+1}|X_{n+1})\in \hat{C}_n(X_{n+1})\textrm{ if and only if }|X_{n+1}|\leq Q_{1-\alpha/2}(\{|Y_i|:i\in\mathcal{I}_2\}).$$

Define $M_R = \#\{i\in\mathcal{I}_2:|X_i|\geq |X_{n+1}|\}$ and
$M_E = \#\{i\in\mathcal{I}_2:|Y_i|\geq |X_{n+1}|\}$. Now, note that
$Q_{1-\alpha/2}(\{|Y_i|:i\in\mathcal{I}_2\})$ is the
$\lceil n_2 (1-\alpha/2)(1+1/n_2)\rceil= \lceil
(1-\alpha/2)(n_2+1)\rceil$
smallest value of $\{|Y_i|:i\in\mathcal{I}_2\}$, which equals the
$n_2+1-\lceil (1-\alpha/2)(n_2+1)\rceil$ largest value. Letting
$m = n_2+1-\lceil (1-\alpha/2)(n_2+1)\rceil$,
$$|X_{n+1}|\leq Q_{1-\alpha/2}(\{|Y_i|:i\in\mathcal{I}_2\})\textrm{ if and only if }M_E\geq m.$$
We now build up the following two lemmas.

\begin{lemma}
\label{Sharpness Comparison Lemma}
For all $0\leq M\leq n_2$, 
$$\mathbb{P}\big\{M_R = M \big\}= \frac{1}{n_2+1}.$$
\end{lemma}

\begin{proof}
This holds from the fact that the values $\{|X_i|:i\in\mathcal{I}_2\}\cup\{|X_{n+1}|\}$ are i.i.d.~and have a distribution over $[0,0.5]$ with no point masses. As a result, $M_R$ is uniformly distributed over $\{0,1,\ldots,n_2\}$.
\end{proof}

\begin{lemma}
\label{Sharpness Probability Lemma}
$M_E|M_R\sim\textnormal{Binom}(M_R,0.5+\delta)$. 
\end{lemma}

\begin{proof}
First, note that for all $i\in\mathcal{I}_2$, if $|X_i|<|X_{n+1}|$, then $\mathbb{P}\{|Y_i|\geq |X_{n+1}|\}=0$. This is because if $|X_i|<|X_{n+1}|$, then $|Y_i|\leq |X_i|< |X_{n+1}|$. Additionally, if $|X_i|\geq|X_{n+1}|$, then $\mathbb{P}\{|Y_i|\geq |X_{n+1}|\}=0.5+\delta$. This is due to the fact that if $|X_i|\geq |X_{n+1}|$, we have that $|Y_i|=0$ with probability $0.5-\delta$ and $|Y_i|=|X_i|$ with probability $0.5+\delta$. With probability $1$, $|X_{n+1}|>0$. Therefore, $|Y_i|\geq |X_{n+1}|$ if and only if $|Y_i|=|X_i|$, which occurs with probability $0.5+\delta$.

Furthermore, the events $\{|Y_i|=|X_{i}|\}$ are mutually independent for
all $i\in\mathcal{I}_2$ (the pairs $(X_i, Y_i)$ are
i.i.d.). Then, since 
$$M_E = \sum_{i\in\mathcal{I}_2} \textbf{1}[|Y_i|\geq |X_{n+1}|] = \sum_{i\in\mathcal{I}_2} \textbf{1}[|X_i|\geq |X_{n+1}|]\textbf{1}[|Y_i|= |X_{i}|]$$
and each term $\textbf{1}[|Y_i|= |X_{i}|]$ is i.i.d.~Bernoulli with
probability $0.5+\delta$, and
$M_R=\sum_{i\in\mathcal{I}_2} \textbf{1}[|X_i|\geq |X_{n+1}|]$, the
result follows.
\end{proof}

We now apply our two lemmas:
\begin{align*}
\mathbb{P}\big\{M_E \geq m \big\} 
=& \sum_{j=0}^{n_2} \mathbb{P}\big\{ M_R = j \big\}\mathbb{P}\big\{M_E \geq m | M_R = j \big\} \\
=& \frac{1}{n_2+1}\sum_{j=0}^{n_2} \sum_{k=m}^j \mathbb{P}\big\{M_E \geq m | M_R = j \big\} \\
=& \frac{1}{n_2+1}\sum_{j=0}^{n_2} \sum_{k=m}^j \dbinom{j}{k}(0.5+\delta)^k(0.5-\delta)^{j-k};
\end{align*}
the second equality follows from Lemma \ref{Sharpness Comparison
  Lemma}, and the last from Lemma \ref{Sharpness Probability
  Lemma}. Now, by applying the same train of logic as used in
Appendices \ref{Conditional Median Proof} and \ref{Conditional
  Quantile Proof}, 
\begin{align*}
\mathbb{P}\big\{M_E \geq m \big\} 
=& \frac{1}{n_2+1}\sum_{j=0}^{n_2} \Big(1- \sum_{k=0}^{m-1} \dbinom{j}{k}(0.5+\delta)^k(0.5-\delta)^{j-k} \Big) \\
=& 1-\frac{1}{n_2+1}\sum_{j=0}^{n_2} \sum_{k=0}^{m-1} \dbinom{j}{k}(0.5+\delta)^k(0.5-\delta)^{j-k} \\
=& 1-\frac{1}{n_2+1}\sum_{k=0}^{m-1} \Big(\frac{0.5+\delta}{0.5-\delta} \Big)^k\sum_{j=0}^{n_2}  \dbinom{j}{k}(0.5-\delta)^{j} \\
=& 1-\frac{1}{n_2+1}\sum_{k=0}^{m-1}
   \Big(\frac{0.5+\delta}{0.5-\delta}
   \Big)^k\Big(\frac{(0.5-\delta)^k}{(0.5+\delta)^{k+1}} -
   \sum_{j=n_2+1}^{\infty}  \dbinom{j}{k}(0.5-\delta)^{j}\Big), 
\end{align*}
where the last equality is from evaluating the generating function $G(\binom{t}{k}; z) = \sum_{t=0}^{\infty} \binom{t}{k} z^{t}$ at $z = 0.5-\delta$.

Finally, in order to bring this into a coherent bound, we expand the equation to bring out the $1-\alpha$ term and isolate the remainder, which we can then show goes to $0$:
\begin{align*}
\mathbb{P}\big\{M_E \geq m \big\} 
=& 1-\frac{1}{n_2+1}\sum_{k=0}^{m-1} \Big(\frac{1}{0.5+\delta} - \Big(\frac{0.5+\delta}{0.5-\delta} \Big)^k\sum_{j=n_2+1}^{\infty}  \dbinom{j}{k}(0.5-\delta)^{j}\Big)\\
=& 1-\frac{m}{n_2+1}\cdot \frac{1}{0.5+\delta}+\sum_{k=0}^{m-1} \Big(\frac{0.5+\delta}{0.5-\delta} \Big)^k\sum_{j=n_2+1}^{\infty}  \dbinom{j}{k}(0.5-\delta)^{j}\\
\leq& 1-\frac{m}{n_2+1}\cdot \frac{1}{0.5+\delta}+\sum_{k=0}^{m-1} \Big(\frac{0.5+\delta}{0.5-\delta} \Big)^k \dbinom{n_2+1}{k}(0.5-\delta)^{n_2+1} \frac{1}{1-(0.5-\delta)\frac{n_2+1}{n_2+1-k}},
\end{align*}
where the inequality arises from upper bounding the summation
$\sum_{j=n_2+1}^{\infty} \binom{j}{k}(0.5-\delta)^{j}$ by
$$\binom{n_2+1}{k}(0.5-\delta)^{n_2+1}\sum_{j=0}^{\infty}
\left(\frac{n_2+1}{n_2+1-k}(0.5-\delta)\right)^{j}$$
using the maximum ratio of consecutive terms. Applying
$m\leq (n_2+1)\alpha/2$ twice gives 
\begin{align*}
\mathbb{P}\big\{M_E \geq m \big\} 
=& 1-\frac{m}{n_2+1}\cdot \frac{1}{0.5+\delta}+\Big(\frac{0.5+\delta}{0.5-\delta} \Big)^{m-1}\frac{1}{1-(0.5-\delta)\frac{n_2+1}{n_2+1-m}}\sum_{k=0}^{m-1}  \dbinom{n_2+1}{k}(0.5-\delta)^{n_2+1}\\
\leq& 1-\frac{m}{n_2+1}\cdot \frac{1}{0.5+\delta}+\Big(\frac{0.5+\delta}{0.5-\delta} \Big)^{m-1}\Big(2+\frac{\alpha}{1-\alpha}\Big)\sum_{k=0}^{m-1}  \dbinom{n_2+1}{k}(0.5-\delta)^{n_2+1} \\
\leq& 1-\alpha+ \frac{2\alpha\delta}{1+2\delta}+\Big(2+\frac{\alpha}{1-\alpha}\Big)\Big(\frac{0.5+\delta}{0.5-\delta} \Big)^{(n_2+1)\alpha/2}\cdot \frac{\sum_{k=0}^{\lfloor (n_2+1)\alpha/2\rfloor}  \dbinom{n_2+1}{k}}{2^{n_2+1}}.
\end{align*}
Because $\alpha<1$, 
$$\frac{\sum_{k=0}^{\lfloor (n_2+1)\alpha/2\rfloor}
  \binom{n_2+1}{k}}{2^{n_2+1}}\to 0 \quad \text{as} \quad
n_2\to\infty;$$
this is due to the fact that
$$\frac{\sum_{k=0}^{ n_2+1} \binom{n_2+1}{k}}{2^{n_2+1}}=1$$ and that
the standard deviation of $\textrm{Binom}(n_2,0.5)$ is
$\mathcal{O}(\sqrt{n_2})$, meaning that
$$\frac{\sum_{k=\lfloor (n_2+1)\alpha/2\rfloor+1}^{n_2+1-\lfloor
    (n_2+1)\alpha/2\rfloor} \binom{n_2+1}{k}}{2^{n_2+1}}\to 1.$$
Furthermore, as $\textrm{Binom}(n_2,0.5)$ approaches a normal
distribution as $n_2\to\infty$ and $\Phi(-c\sqrt{n_2})$ is
$\mathcal{O}(d^{-n_2})$ for some $d>1$, for small enough $\delta$,
$$\left(\frac{0.5+\delta}{0.5-\delta} \right)^{(n_2+1)\alpha/2}\, 
\frac{\sum_{k=0}^{\lfloor (n_2+1)\alpha/2\rfloor}
  \binom{n_2+1}{k}}{2^{n_2+1}}\to 0 \quad \text{as} \quad
n_2\to\infty.$$
Thus, we can pick $D$ and $N$ such that for all $\delta<D$ and
$n\geq N$,
$$\left(\frac{0.5+\delta}{0.5-\delta} \right)^{(n_2+1)\alpha/2}\,
\frac{\sum_{k=0}^{\lfloor (n_2+1)\alpha/2\rfloor}
  \binom{n_2+1}{k}}{2^{n_2+1}}\leq \epsilon/2,$$
noting that $n_2 = n/2$. Then, setting
$\delta = \min\{\frac{\epsilon}{4\alpha-2\epsilon},D\}$ and 
$\frac{2\alpha\delta}{1+2\delta}\leq \epsilon/2$ yields 
$$\mathbb{P}\big\{M_E \geq m \big\}\leq 1-\alpha+\epsilon/2+\epsilon/2=1-\alpha+\epsilon.$$
This says that the probability $\mathbb{P}\{M_E \geq m \}$ of the
confidence interval containing the conditional median is at most
$1-\alpha+\epsilon$.

\subsection{Proof of Theorem \ref{Crazy Regression Coverage Theorem}}
\label{Crazy Regression Coverage Proof}

We show that given $\epsilon$ there exists $c$, $N$, and $n_1+n_2=n$ for all $n>N$
  such that running Algorithm \ref{Conditional Median Interval} on $n>N$ datapoints from an arbitrary distribution $P$ with regression function $\hat{\mu}_c$ and split sizes $n_1 + n_2 = n$ results in a finite confidence interval that contains the conditional median with probability at least $1-\alpha/2 - \epsilon$. 
  
  For each $x$ in the
  support of $P$, define $m(x) = \textrm{Median}(Y|X=x)$ and recall that $M = \displaystyle \max_{i\in\mathcal{I}_1} |Y_i|$.  We begin
  with two lemmas.

\begin{lemma}
\label{Crazy Regression First Lemma}
For all $i\in\mathcal{I}_2$,
$\mathbb{P}\{|Y_i|\leq M\} \geq 1-\frac{1}{n_1 + 1}$.
\end{lemma}
\begin{proof}
This results from the fact that $|Y_i|$ is exchangeable with $|Y_j|$ for all $j\in\mathcal{I}_1$; thus, the probability that $|Y_i|$ is the unique maximum of the set $\{Y_j:j\in\mathcal{I}_1\cup\{i\}\}$ is bounded above by $\frac{1}{n_1+1}$. Taking the complement yields the desired result.
\end{proof}
\begin{lemma}
\label{Crazy Regression Second Lemma}
For all $i\in\mathcal{I}_2\cup\{n+1\}$,
$\mathbb{P}\{|m(X_i)|\leq M\}\geq 1-\frac{2}{n_1+1}$.
\end{lemma}
\begin{proof}
  Note that $|m(X_i)|$ is exchangeable with $|m(X_j)|$ for all
  $j\in\mathcal{I}_1$. Letting
  $M_R = \#\{|m(X_j)|\geq |m(X_i)|: j\in\mathcal{I}_1\}$,
  exchangeability gives that the CDF of $M_R$ is bounded below by the
  CDF of the uniform distribution over $\{0,1,\ldots,n_1\}$.  For each
  $j\in\mathcal{I}_1$, the event $\{|Y_j|\geq |m(X_j)|\}$ occurs with
  probability at least $1/2$ by definition of the median; moreover,
  these events are mutually independent. 
  {Therefore, if we condition on $M_R$, we have that $$\mathbb{P}\{|m(X_i)| > \displaystyle \max_{j\in\mathcal{I}_1}
  |Y_j|\big| M_R = k\}\leq\prod_{\substack{j\in\mathcal{I}_1 \\ |m(X_j)|\geq |m(X_i)|}} \mathbb{P}\{|Y_j| < |m(X_j)|\}  \leq 2^{-k}. $$}
    
    Putting this
  together, we see that
\begin{align*}
    \mathbb{P}\{|m(X_i)| > \max_{j\in\mathcal{I}_1} |Y_j|\} 
    =& \sum_{k=0}^{n_1} \mathbb{P}\{M_R = k\}\mathbb{P}\{|m(X_i)| > \max_{j\in\mathcal{I}_1} |Y_j| \big| M_R = k\} \\
    \leq& \sum_{k=0}^{n_1} \mathbb{P}\{M_R = k\} \frac{1}{2^k} \\
    \leq& \frac{1}{n_1 + 1}\sum_{k=0}^{n_1} \frac{1}{2^k} \\ 
    \leq& \frac{2}{n_1 + 1}.
\end{align*}
Taking the complement yields the desired result.
\end{proof}

Let $A$ be the event
$\{|Y_i|\leq M\textrm{ for all }i\in\mathcal{I}_2\textrm{ and
}|m(X_i)|\leq M\textrm{ for all }i\in\mathcal{I}_2\cup\{n+1\}\}$.  By
Lemmas \ref{Crazy Regression First Lemma} and \ref{Crazy Regression
  Second Lemma},
$\mathbb{P}\{A\}\geq 1 - \displaystyle\frac{3n_2 + 2}{n_1 +
  1}$. Select $N = \displaystyle\Big\lfloor\frac{12/\alpha + 10}{\epsilon}\Big\rfloor + \lfloor 2/\alpha\rfloor + 1$, and for all $n> N$, set $n_2 = \lfloor 2/\alpha\rfloor + 1$ and $n_1 = n - n_2$. {As a result, we have that $1/n_2 < \alpha/2$ and $\displaystyle\frac{3n_2 + 2}{n_1 + 1} < \epsilon/2$, so $\mathbb{P}\{A\}\geq 1-\epsilon/2$.}

Next, let $B$ be the event
$\displaystyle\{|\hat{\mu}_c(X_{i_1}) - \hat{\mu}_c(X_{i_2})| > 2M
\textrm{ for all }i_1\neq i_2\in\mathcal{I}_2\cup\{n+1\}\}$. Note that
$\displaystyle\lim_{c\to\infty}\mathbb{P}\{B\} = 1$ by definition of
$\hat{\mu}_c$. Select $c$ such that
$\mathbb{P}\{B\} \geq 1 - \epsilon/2$. By the union bound,
$\mathbb{P}\{A\cap B\} \geq 1-\epsilon$.

\begin{lemma}
\label{Crazy Regression Third Lemma}
{On the event $A\cap B$}, for all $i\in\mathcal{I}_2$,
$$|m(X_i) - \hat{\mu}_c(X_i)| \geq |m(X_{n+1}) - \hat{\mu}_c(X_{n+1})| \quad \text{if and only if} \quad |Y_i - \hat{\mu}_c(X_i)| \geq |m(X_{n+1}) - \hat{\mu}_c(X_{n+1})|.$$
\end{lemma}
\begin{proof}
  Notice that on the event $A \cap B$,
$|m(X_i) - \hat{\mu}_c(X_i)|,|Y_i - \hat{\mu}_c(X_i)|\in
[|\hat{\mu}_c(X_i)| - M, |\hat{\mu}_c(X_i)| + M]$. This holds because
$|m(X_i)|,|Y_i|\leq M$ on the event $A$. Similarly,
$|m(X_{n+1}) - \hat{\mu}_c(X_{n+1})| \in [|\hat{\mu}_c(X_{n+1})| - M,
|\hat{\mu}_c(X_{n+1})| + M]$. These two intervals both have length
$2M$, but their centers are at a distance greater than $2M$ on the
event $B$, meaning that the intervals are disjoint. Therefore,
$|m(X_i) - \hat{\mu}_c(X_i)| \geq |m(X_{n+1}) - \hat{\mu}_c(X_{n+1})|$
implies that all elements of the first interval are greater than all
elements of the second, so
$|Y_i - \hat{\mu}_c(X_i)| \geq |m(X_{n+1}) - \hat{\mu}_c(X_{n+1})|$;
similarly,
$|Y_i - \hat{\mu}_c(X_i)| \geq |m(X_{n+1}) - \hat{\mu}_c(X_{n+1})|$
also implies that all elements of the first interval are greater than
all elements of the second, so
$|m(X_i) - \hat{\mu}_c(X_i)| \geq |m(X_{n+1}) -
\hat{\mu}_c(X_{n+1})|$.
\end{proof}

Looking at Algorithm \ref{Conditional Median Interval}, we have that
$m(X_{n+1})\in\hat{C}_n(X_{n+1})$ if
$|m(X_{n+1}) - \hat{\mu}_c(X_{n+1})| \leq Q_{1-\alpha/2}(E)$, where
$E_i = |Y_i - \hat{\mu}_c(X_i)|$ for all $i\in\mathcal{I}_2$. Because
$1/n_2 < \alpha/2$, $Q_{1-\alpha/2}(E)$ is finite and thus the
confidence interval is bounded. By Lemma \ref{Crazy Regression Third
  Lemma}, on the event $A \cap B$,
$|m(X_{n+1}) - \hat{\mu}_c(X_{n+1})| \leq Q_{1-\alpha/2}(E)$ if and
only if $|m(X_{n+1}) - \hat{\mu}_c(X_{n+1})| \leq Q_{1-\alpha/2}(F)$,
where $F_i = |m(X_i) - \hat{\mu}_c(X_i)|$ for all
$i\in\mathcal{I}_2$. 

{Define $C$ to be the event $\displaystyle\{ |m(X_{n+1}) - \hat{\mu}_c(X_{n+1})| \leq Q_{1-\alpha/2}(F)\}$. We have just shown that on the event $A\cap B\cap C$, we have $m(X_{n+1})\in\hat{C}_n(X_{n+1})$. Additionally, because the elements of $\{|m(X_i) - \hat{\mu}_c(X_i)|:i\in\mathcal{I}_2\cup\{n+1\}\}$ are exchangeable, we have that $\mathbb{P}\{C\}\geq 1-\alpha/2$. Then, by the union bound,
$$\mathbb{P}\{m(X_{n+1})\in\hat{C}_n(X_{n+1})\}\geq \mathbb{P}\{A\cap B\cap C\} \geq 1-\alpha/2 -\epsilon,$$
proving the desired result.}

\subsection{Proving Extensions of Theorem \ref{All Predictive Algorithms}}
\label{All Predictive Intervals Extension Proof}
We show the following two results:

Let $\hat{C}_n(x) = [\hat{L}_n(x), \hat{H}_n(x)]$ be any algorithm that only outputs confidence intervals and satisfies $\mathbb{P}\{Y_{n+1} \geq \hat{L}_n(X_{n+1})\}\geq 1-rq$ and $\mathbb{P}\{Y_{n+1} \leq \hat{H}_n(X_{n+1})\}\geq 1-s(1-q)$ for some $r+s=\alpha$. Then, $\hat{C}_n$ satisfies $(1-\alpha,q)$-Quantile. 
Secondly, if $\hat{C}_n$ only outputs confidence intervals and satisfies $(1-\min\{q,1-q\}\alpha)$-Predictive, then it satisfies $(1-\alpha,q)$-Quantile. 

Consider a distribution $P$, and for all $x\in\mathbb{R}^d$ in the
support of $P$, let $q(x)$ be $\textrm{Quantile}_q(Y|X=x)$. We prove the first result in two parts.

Conditioning on whether or not $q(X_{n+1})$ is greater than or equal to $\hat{L}(X_{n+1})$, note that
\begin{align*}
1-rq\leq & \mathbb{P}\{Y_{n+1} \geq \hat{L}_n(X_{n+1})\} \\
=& \mathbb{P}\{Y_{n+1} \geq \hat{L}_n(X_{n+1})|q(X_{n+1})\geq \hat{L}_n(X_{n+1})\}\mathbb{P}\{q(X_{n+1})\geq \hat{L}_n(X_{n+1})\}\\
&+ \mathbb{P}\{Y_{n+1} \geq \hat{L}_n(X_{n+1})|q(X_{n+1})< \hat{L}_n(X_{n+1})\}\mathbb{P}\{q(X_{n+1})< \hat{L}_n(X_{n+1})\}\\
\leq& \mathbb{P}\{q(X_{n+1})\geq \hat{L}_n(X_{n+1})\}+ (1-q)\mathbb{P}\{q(X_{n+1})< \hat{L}_n(X_{n+1})\}\\
=& (1-q) + q\cdot\mathbb{P}\{q(X_{n+1})\geq \hat{L}_n(X_{n+1})\}
\end{align*}
where we use the fact that if $q(X_{n+1})< \hat{L}_n(X_{n+1})$, at most $1-q$ of the conditional distribution of $Y|X=X_{n+1}$ can be at least $\hat{L}_n(X_{n+1})$. Subtracting $1-q$ from both sides and dividing by $q$ tells us that $1-r\leq \mathbb{P}\{q(X_{n+1})\geq \hat{L}_n(X_{n+1})\}$.

Similarly,
\begin{align*}
1-s(1-q)\leq & \mathbb{P}\{Y_{n+1} \leq \hat{H}_n(X_{n+1})\} \\
=& \mathbb{P}\{Y_{n+1} \leq \hat{H}_n(X_{n+1})|q(X_{n+1})\leq \hat{H}_n(X_{n+1})\}\mathbb{P}\{q(X_{n+1})\leq \hat{H}_n(X_{n+1})\}\\
&+ \mathbb{P}\{Y_{n+1} \leq \hat{H}_n(X_{n+1})|q(X_{n+1})> \hat{H}_n(X_{n+1})\}\mathbb{P}\{q(X_{n+1})> \hat{H}_n(X_{n+1})\}\\
\leq& \mathbb{P}\{q(X_{n+1})\leq \hat{H}_n(X_{n+1})\}+ q\cdot \mathbb{P}\{q(X_{n+1})> \hat{H}_n(X_{n+1})\}\\
=& q + (1-q)\mathbb{P}\{q(X_{n+1})\leq \hat{H}_n(X_{n+1})\}.
\end{align*}
Subtracting $q$ from both sides and dividing by $1-q$ tells us that $1-s\leq \mathbb{P}\{q(X_{n+1})\leq \hat{H}_n(X_{n+1})\}$.

Then, by the union bound, we have that $1-\alpha = 1-(r+s)\leq \mathbb{P}\{q(X_{n+1})\geq \hat{L}_n(X_{n+1})\cap q(X_{n+1})\leq \hat{H}_n(X_{n+1})\} = \mathbb{P}\{q(X_{n+1}\in\hat{C}_n(X_{n+1})\}$, proving the first result.

We can prove the second result in the exact same fashion as the proof of Theorem \ref{All Predictive Algorithms}. We have that
\begin{align*}
    1-\min\{q,1-q\}\alpha \leq &\mathbb{P}\{Y_{n+1}\in\hat{C}_n(X_{n+1}) \} \\
    =& \mathbb{P}\{Y_{n+1}\in\hat{C}_n(X_{n+1}) | q(X_{n+1})\in\hat{C}_n(X_{n+1})\}\mathbb{P}\{ q(X_{n+1})\in\hat{C}_n(X_{n+1})\}\\
    &+ \mathbb{P}\{Y_{n+1}\in\hat{C}_n(X_{n+1}) | q(X_{n+1})\not\in\hat{C}_n(X_{n+1})\}\mathbb{P}\{ q(X_{n+1})\not\in\hat{C}_n(X_{n+1})\} \\
    \leq& \mathbb{P}\{ q(X_{n+1})\in\hat{C}_n(X_{n+1})\}+ \max\{q,1-q\}\mathbb{P}\{ q(X_{n+1})\not\in\hat{C}_n(X_{n+1})\} \\
    =& \max\{q,1-q\} + \min\{q,1-q\}\mathbb{P}\{ q(X_{n+1})\in\hat{C}_n(X_{n+1})\}.
\end{align*}
Subtracting $\max\{q,1-q\}$ and dividing by $\min\{q,1-q\}$ yields the desired result.

\section{Additional Results}

\subsection{Alternate Proof of Theorem \ref{Conditional Median Interval}}

The proof of the theorem relies on two lemmas: the first establishes a
connection between $\textrm{Median}(Y_{n+1}|X_{n+1})$ and
$\textrm{Median}(Y|X=X_i)$'s using exchangeability, and the second
gives us a relationship between $\textrm{Median}(Y|X=X_i)$'s and
the $Y_i$'s.

We begin with some notation. For all $x\in\mathbb{R}^d$ in the
support of $P$, set
$m(x) = \textrm{Median}(Y|X=x)$.  Also, for $1\leq i\leq n+1$, let
$R(X_i) = |m(X_i)-\hat{\mu}(X_i)|$, and for $1\leq i\leq n$ let
$E(X_i) = |Y_i - \hat{\mu}(X_i)|$.  Finally, put
$M_R = \#\{i\in\mathcal{I}_2:R(X_i)\geq R(X_{n+1})\}$ as the number of
$i\in\mathcal{I}_2$ for which $R(X_i)\geq R(X_{n+1})$, and
$M_E = \#\{i\in\mathcal{I}_2:E(X_i)\geq R(X_{n+1})\}$ as the number of
$i\in\mathcal{I}_2$ for which $E(X_i)\geq R(X_{n+1})$.

Note that
$\textrm{Median}(Y_{n+1}|X_{n+1})\in \hat{C}_n(X_{n+1}) =
[\hat{\mu}(X_{n+1})-Q_{1-\alpha/2}(E),\hat{\mu}(X_{n+1})+Q_{1-\alpha/2}(E)]$
if and only if
$|\textrm{Median}(Y_{n+1}|X_{n+1}) - \hat{\mu}(X_{n+1})|=R(X_{n+1})$
is at most $Q_{1-\alpha/2}(E)$. Thus, we must study the value of
$R(X_{n+1})$ in relation to the elements of
$\{E(X_i):i\in\mathcal{I}_2\}$.

The first lemma relates $R(X_{n+1})$ to the other $R(X_i)$'s.
\begin{lemma}
\label{Median Comparison Lemma}
For all $0\leq m\leq n_2$, 
$$\mathbb{P}\big\{M_R \geq m \big\}\geq 1-\frac{m}{n_2+1}.$$
\end{lemma}
\begin{proof}
  Our statement follows from the fact that our samples are
  i.i.d. Because $\hat{\mu}$ is independent of $(X_i,Y_i)$ for
  $i\in\mathcal{I}_2$ and independent of $X_{n+1}$, the values
  $\{R(X_i):i\in\mathcal{I}_2\}\cup\{R(X_{n+1})\}$ are i.i.d.~as
  well. Then, the probability of less than $m$ values in
  $\{R(X_i):i\in\mathcal{I}_2\}$ being at least $R(X_{n+1})$ is
  bounded above by $\frac{m}{|\mathcal{I}_2|+1}$, as the ordering of
  these values is uniformly random. Taking the complement of both
  sides gives the result.
\end{proof}

The second lemma gives a direct relationship between $E(X_i)$ and
$R(X_i)$. (Note that the events $\{E(X_i)\geq R(X_i)\}$ below are
mutually independent.) 
\begin{lemma}
\label{Median Probability Lemma}
For all $i\in\mathcal{I}_2$, 
$$\mathbb{P}\big\{E(X_i)\geq R(X_i) \big\}\geq 1/2.$$
\end{lemma}
\begin{proof}
We see that $\mathbb{P}\{Y_i\geq m(X_i)\}\geq
1/2$ and $\mathbb{P}\{Y_i\leq m(X_i)\}\geq 1/2$ by the definition of
the conditional median. 
Furthermore, the events $\{Y_i\geq m(X_i)\}$ and
$\{Y_i\leq m(X_i)\}$ are independent of the events $\{m(X_i)\geq \hat{\mu}(X_i)\}$ and $\{m(X_i)\leq \hat{\mu}(X_i)\}$ given $X_i$, as
$\hat{\mu}$ is a function of $\{(X_i,Y_i): i\in \mathcal{I}_1\}$ and the datapoints are i.i.d. Then, conditioned on $X_i$, if $m(X_i)\geq \hat{\mu}(X_i)$, with probability $1/2$ we have that $m(X_i)\leq Y_i$, in which case 
$|m(X_i)-\hat{\mu}(X_i)| = m(X_i)-\hat{\mu}(X_i)\leq Y_i -
\hat{\mu}(X_i) = |Y_i - \hat{\mu}(X_i)|$. Similarly, conditioned on $X_i$ again, if $m(X_i)<
\hat{\mu}(X_i)$, with probability $1/2$ we have that $m(X_i)\geq Y_i$,
in which case $|m(X_i)-\hat{\mu}(X_i)| = \hat{\mu}(X_i)-m(X_i)\leq
\hat{\mu}(X_i) -Y_i= |Y_i - \hat{\mu}(X_i)|$. The conclusion holds conditionally in both cases; marginalizing out $X_i$ yields the desired result.
\end{proof}

We now study the number of datapoints obeying $E(X_i)\geq R(X_{n+1})$
by combining these lemmas together. Consider any $0\leq m\leq
n_2$. Note that by conditioning on $M_R$,
\begin{align*}
\mathbb{P}\big\{M_E \geq m \big\}
=& \sum_{j=0}^{n_2} \mathbb{P}\big\{ M_R = j \big\}\mathbb{P}\big\{M_E \geq m | M_R = j \big\} \\
\geq& \sum_{j=0}^{n_2} \mathbb{P}\big\{ M_R = j \big\}\sum_{k=m}^j \dbinom{j}{k}2^{-j}\\
\geq& \sum_{j=0}^{n_2} \frac{1}{n_2+1}\sum_{k=m}^j \dbinom{j}{k}2^{-j}.
\end{align*}
The first inequality holds true by Lemma \ref{Median Probability
  Lemma}, which implies that $\mathbb{P}\{M_E \geq m | M_R = j \}$ can
be bounded below by the probability that $M\geq m$ for
$M\sim\textrm{Binom}(j,0.5)$. The second inequality is due to Lemma
\ref{Median Comparison Lemma}. We know that
$\sum_{k=m}^j \binom{j}{k}2^{-j}$ is a nondecreasing function of $j$;
by Lemma \ref{Median Comparison Lemma}, the CDF of the distribution of
$M_R$ is lower bounded by the CDF of the uniform distribution over
$\{0,1,\ldots,n_2\}$, meaning that
$$\mathbb{E}_{M_R}\left[\sum_{k=m}^{M_R} \binom{M_R}{k}2^{-M_R}\right]\geq
\frac{1}{n_2+1}\sum_{j=0}^{n_2}\sum_{k=m}^j \binom{j}{k}2^{-j}.$$
This gives
\begin{align*}
\sum_{j=0}^{n_2} \frac{1}{n_2+1}\sum_{k=m}^j \dbinom{j}{k}2^{-j}
=& \sum_{j=0}^{n_2} \frac{1}{n_2+1}\Big(1- \sum_{k=0}^{m-1} \dbinom{j}{k}2^{-j} \Big) \\
=& 1-\frac{1}{n_2+1}\sum_{j=0}^{n_2} \sum_{k=0}^{m-1} \dbinom{j}{k}2^{-j}  \\
=& 1-\frac{1}{n_2+1}\sum_{k=0}^{m-1} \sum_{j=0}^{n_2}  \dbinom{j}{k}2^{-j}  \\
\geq& 1-\frac{1}{n_2+1}\sum_{k=0}^{m-1} \sum_{j=0}^{\infty}  \dbinom{j}{k}2^{-j}  \\
=& 1-\frac{1}{n_2+1}\sum_{k=0}^{m-1} 2 &\\
=& 1-\frac{2m}{n_2+1}.
\end{align*}
The second-to-last equality comes from evaluating the generating
function $G(\binom{t}{k}; z) = \sum_{t=0}^{\infty} \binom{t}{k} z^{t}$ at $z = 0.5$.

Putting it together, in Algorithm \ref{Conditional Median Interval},
$Q_{1-\alpha/2}(E)$ is set to be the $(1-\alpha/2)(1+1/n_2)$-th
quantile of $\{E_i:i\in\mathcal{I}_2\}$. This is equal to the
$\lceil n_2 (1-\alpha/2)(1+1/n_2) \rceil = \lceil (1-\alpha/2)(n_2+1)
\rceil$
smallest value of $\{E_i:i\in\mathcal{I}_2\}$. This means that if
$M_E\geq n_2 - \lceil (1-\alpha/2)(n_2+1) \rceil+1$, then $R(X_{n+1})$
will be at most the $n_2 - \lceil (1-\alpha/2)(n_2+1) \rceil+1$
largest value of $\{E_i:i\in\mathcal{I}_2\}$, which is equal to
the $\lceil (1-\alpha/2)(n_2+1) \rceil$ smallest value, or
$Q_{1-\alpha/2}(E)$.

However, we have calculated a lower bound for the inverse CDF of $M_E$
earlier. Substituting this in, we get that
\begin{align*}
\mathbb{P}\{ R(X_{n+1}) \leq  Q_{1-\alpha/2}(E)\}
\geq& \mathbb{P}\{ M_E \geq n_2 - \lceil (1-\alpha/2)(n_2+1) \rceil+1\}\\
\geq& \mathbb{P}\{ M_E \geq n_2 +1 -(1-\alpha/2)(n_2+1) \} \\
=& \mathbb{P}\{ M_E \geq \alpha/2(n_2 +1) \} \\
\geq& 1-\alpha
\end{align*}
by our previous calculation, completing our proof.

\subsection{Alternate Proof of Theorem \ref{Conditional Quantile Interval}}

Our approach is similar to that in Appendix \ref{Conditional Median
  Proof}. 
The main difference in the proof arises from the fact that Algorithm
\ref{Conditional Quantile Interval} no longer uses the absolute value
and uses two separate fitted functions, meaning that it is important
to bound the probability of the confidence interval covering the
desired value from both sides.

We begin with some definitions. For all $x\in\mathbb{R}^d$ in the
support of $P$, let $q(x)$ be $\textrm{Quantile}_q(Y|X=x)$. For all
$1\leq i\leq n+1$, define:
\begin{itemize}
\itemsep0em 
\item $R^{\textnormal{lo}}(X_i) = f^{\textnormal{lo}}(X_i,q(X_i))$ and $R^{\textnormal{hi}}(X_i) = f^{\textnormal{hi}}(X_i,q(X_i))$. 
\item $E^{\textnormal{lo}}(X_i) = f^{\textnormal{lo}}(X_i,Y_i)$
  and $E^{\textnormal{hi}}(X_i) = f^{\textnormal{hi}}(X_i,Y_i)$.
\item
  $M_R^{\textrm{lo}} = \#\{i\in\mathcal{I}_2:
  R^{\textnormal{lo}}(X_i)\leq R^{\textnormal{lo}}(X_{n+1}) \}$
  and
  $M_R^{\textnormal{hi}} = \#\{i\in\mathcal{I}_2:
  R^{\textnormal{hi}}(X_i)\geq R^{\textnormal{hi}}(X_{n+1}) \}$
  as the number of $i\in\mathcal{I}_2$ for which
  $R^{\textnormal{lo}}(X_i)\leq R^{\textnormal{lo}}(X_{n+1})$ and
  $R^{\textnormal{hi}}(X_i)\geq R^{\textnormal{hi}}(X_{n+1})$
  respectively.
\item
  $M_E^{\textrm{lo}} = \#\{i\in\mathcal{I}_2:
  E^{\textnormal{lo}}(X_i)\leq R^{\textnormal{lo}}(X_{n+1}) \}$
  and
  $M_E^{\textnormal{hi}} = \#\{i\in\mathcal{I}_2:
  E^{\textnormal{hi}}(X_i)\geq R^{\textnormal{hi}}(X_{n+1}) \}$
  as the number of $i\in\mathcal{I}_2$ for which
  $E^{\textnormal{lo}}(X_i)\leq R^{\textnormal{lo}}(X_{n+1})$ and
  $E^{\textnormal{hi}}(X_i)\geq R^{\textnormal{hi}}(X_{n+1})$
  respectively.
\end{itemize}

Now, note that
$\textrm{Quantile}_q(Y|X=X_{n+1})\in \hat{C}_n(X_{n+1})$ precisely
when
$f^{\textnormal{lo}}(X_{n+1},\textrm{Quantile}_q(Y|X=X_{n+1}))=R^{\textnormal{lo}}(X_{n+1})\geq
Q^{\textnormal{lo}}_{rq}$
and
$f^{\textnormal{hi}}(X_{n+1},\textrm{Quantile}_q(Y|X=X_{n+1}))=R^{\textnormal{hi}}(X_{n+1})\leq
Q^{\textnormal{hi}}_{1-s(1-q)}$.
As such, we proceed to develop two lemmas which extend those from
Section \ref{Conditional Median Proof}: the first helps us to
understand the distributions of $M_R^{\textnormal{hi}}$ and
$M_R^{\textrm{lo}}$, and the second studies the individual
relationships between $E^{\textnormal{lo}}(X_i)$ and
$R^{\textnormal{lo}}(X_i)$ and between
$E^{\textnormal{hi}}(X_i)$ and $R^{\textnormal{hi}}(X_i)$. With both
of these lemmas, we are able to bound the probability of each event
$\{R^{\textnormal{lo}}(X_{n+1})\geq Q^{\textnormal{lo}}_{rq}\}$ and
$\{R^{\textnormal{hi}}(X_{n+1})\leq Q^{\textnormal{hi}}_{1-s(1-q)}\}$.

\begin{lemma}
\label{Quantile Comparison Lemma}
For all $0\leq m\leq n_2$, 
$$\mathbb{P}\big\{M_R^{\textnormal{lo}} \geq m \big\}\geq 1-\frac{m}{n_2+1}$$
and
$$\mathbb{P}\big\{M_R^{\textnormal{hi}} \geq m \big\}\geq 1-\frac{m}{n_2+1}.$$
\end{lemma}

\begin{proof}
  We prove the result for $M_R^{\textnormal{hi}}$; the same approach
  holds for $M_R^{\textrm{lo}}$. Because $f^{\textnormal{hi}}$ is
  independent of $(X_i,Y_i)$ for $i\in\mathcal{I}_2$ and independent
  of $X_{n+1}$, the values in 
  $\{R^{\textnormal{hi}}(X_i):i\in\mathcal{I}_2\}\cup\{R^{\textnormal{hi}}(X_{n+1})\}$
  are i.i.d.. Thus, the probability of less than $m$ values in
  $\{R^{\textnormal{hi}}(X_i):i\in\mathcal{I}_2\}$ being at least
  $R^{\textnormal{hi}}(X_{n+1})$ is bounded above by
  $\frac{m}{|\mathcal{I}_2|+1}$, as the ordering of these values is
  uniformly random. Taking the complement of both sides establishes
  the claim.
\end{proof}

\begin{lemma}
\label{Quantile Probability Lemma}
For all $i\in\mathcal{I}_2$, 
$$\mathbb{P}\big\{E^{\textnormal{lo}}(X_i)\leq R^{\textnormal{lo}}(X_i) \big\}\geq q$$
and
$$\mathbb{P}\big\{E^{\textnormal{hi}}(X_i)\geq R^{\textnormal{hi}}(X_i) \big\}\geq 1-q.$$
\end{lemma}
\begin{proof}
  We see that $\mathbb{P}\{Y_i\leq q(X_i)\}\geq q$ and
  $\mathbb{P}\{Y_i\geq q(X_i)\}\geq 1-q$ by the definition of the
  conditional quantile. 
  Then, 
  with probability at least $q$ we have that
  $E^{\textnormal{lo}}(X_i) = f^{\textnormal{lo}}(X_i,Y_i)\leq
  f^{\textnormal{lo}}(X_i,q(X_i))=R^{\textnormal{lo}}(X_i)$
  by the definition of a locally nondecreasing conformity
  score. Similarly, 
  with probability at least $1-q$ we have that
  $E^{\textnormal{hi}}(X_i) = f^{\textnormal{hi}}(X_i,Y_i)\geq
  f^{\textnormal{hi}}(X_i,q(X_i))=R^{\textnormal{hi}}(X_i)$,
  thereby concluding the proof.
\end{proof}

We now study the number of $i\in\mathcal{I}_2$ such that 
$E^{\textnormal{hi}}(X_i)\geq R^{\textnormal{hi}}(X_{n+1})$ by
combining these two lemmas together.  Consider any $0\leq m\leq n_2$,
and note that by conditioning on $M_R^{\textnormal{hi}}$, we have 
\begin{align*}
\mathbb{P}\big\{M_E^{\textnormal{hi}} \geq m \big\}
=& \sum_{j=0}^{n_2} \mathbb{P}\big\{ M_R^{\textnormal{hi}} = j \big\}\mathbb{P}\big\{M_E^{\textnormal{hi}} \geq m | M_R^{\textnormal{hi}} = j \big\} \\
\geq& \sum_{j=0}^{n_2} \mathbb{P}\big\{ M_R^{\textnormal{hi}} = j \big\}\sum_{k=m}^j \dbinom{j}{k}(1-q)^kq^{j-k}\\
\geq& \sum_{j=0}^{n_2} \frac{1}{n_2+1}\sum_{k=m}^j \dbinom{j}{k}(1-q)^kq^{j-k}.
\end{align*}
The first inequality holds because
$\mathbb{P}\{M_E^{\textnormal{hi}} \geq m | M_R^{\textnormal{hi}} = j
\}$
can be bounded below by the probability that $M\geq m$ for
$M\sim\textrm{Binom}(j,1-q)$ by Lemma \ref{Quantile Probability
  Lemma}. The second inequality holds since the CDF $M_R^{\textnormal{hi}}$ is greater than or
equal to the CDF of the uniform distribution over $\{0,1,\ldots,n_2\}$
by Lemma \ref{Quantile Comparison Lemma}. Then, as
$\sum_{k=m}^j \binom{j}{k}(1-q)^kq^{j-k}$ is a nondecreasing function
of $j$, we have
$$\mathbb{E}_{M_R^{\textnormal{hi}}}\left[\sum_{k=m}^{M_R^{\textnormal{hi}}}
\binom{M_R^{\textnormal{hi}}}{k}(1-q)^kq^{M_R^{\textnormal{hi}}-k}\right]\geq
\frac{1}{n_2+1}\sum_{j=0}^{n_2}\sum_{k=m}^j
\binom{j}{k}(1-q)^kq^{j-k}.$$
We now solve the summation, which gives 
\begin{align*}
\sum_{j=0}^{n_2} \frac{1}{n_2+1}\sum_{k=m}^j \dbinom{j}{k}(1-q)^kq^{j-k}
=& \sum_{j=0}^{n_2} \frac{1}{n_2+1}\Big(1- \sum_{k=0}^{m-1} \dbinom{j}{k}(1-q)^kq^{j-k} \Big) \\
=& 1-\frac{1}{n_2+1}\sum_{j=0}^{n_2} \sum_{k=0}^{m-1} \dbinom{j}{k}(1-q)^kq^{j-k} \\
=& 1-\frac{1}{n_2+1}\sum_{k=0}^{m-1} \sum_{j=0}^{n_2}  \dbinom{j}{k}(1-q)^kq^{j-k}  \\
\geq& 1-\frac{1}{n_2+1}\sum_{k=0}^{m-1} \Big(\frac{1-q}{q}\Big)^k\sum_{j=0}^{\infty}  \dbinom{j}{k}q^{j}  \\
=& 1-\frac{1}{n_2+1}\sum_{k=0}^{m-1} \frac{1}{1-q} &\\
=& 1-\frac{1}{1-q}\cdot\frac{m}{n_2+1}, 
\end{align*}
where the second-to-last equality is from evaluating the generating function $G(\binom{t}{k}; z) = \sum_{t=0}^{\infty} \binom{t}{k} z^{t}$ at $z = q$. 

Note that this same calculation works for counting the
$i\in\mathcal{I}_2$ with
$E^{\textnormal{lo}}(X_i)\leq R^{\textnormal{lo}}(X_{n+1})$ using the
same lemmas, substituting $M_E^{\textrm{lo}}$ for
$M_E^{\textnormal{hi}}$, $M_R^{\textrm{lo}}$ for
$M_R^{\textnormal{hi}}$, and $q$ for $1-q$ within the
calculation. This gives
$$\mathbb{P}\big\{M_E^{\textnormal{lo}} \geq m \big\}\geq 1-\frac{1}{q}\cdot\frac{m}{n_2+1}.$$

Now, in Algorithm \ref{Conditional Quantile Interval},
$Q_{1-s(1-q)}^{\textnormal{hi}}(E)$ is defined as the
$(1-s(1-q))(1+1/n_2)$-th quantile of
$\{E^{\textnormal{hi}}_i:i\in\mathcal{I}_2\}$. This is equal to the
$\lceil n_2(1-s(1-q))(1+1/n_2)\rceil = \lceil (1-s(1-q))(n_2+1)\rceil$
smallest value of
$\{E^{\textnormal{hi}}_i:i\in\mathcal{I}_2\}$. Thus, if
$M_E^{\textnormal{hi}}\geq n_2 - \lceil (1-s(1-q))(n_2+1)\rceil+1$,
then $R^{\textnormal{hi}}(X_{n+1})$ will be at most the
$n_2 - \lceil (1-s(1-q))(n_2+1)\rceil+1$ largest value of
$\{E^{\textnormal{hi}}_i:i\in\mathcal{I}_2\}$, which is equal to the
$\lceil (1-s(1-q))(n_2+1)\rceil$ smallest value, or
$Q_{1-s(1-q)}^{\textnormal{hi}}(E)$.  Then, using our earlier lower
bound for the inverse CDF of $M_E^{\textnormal{hi}}$, we get that
\begin{align*}
\mathbb{P}\{ R^{\textnormal{hi}}(X_{n+1}) \leq  Q_{1-s(1-q)}^{\textnormal{hi}}(E)\}
\geq& \mathbb{P}\{ M_E^{\textnormal{hi}} \geq n_2 - \lceil (1-s(1-q))(n_2+1)\rceil+1\}\\
\geq& \mathbb{P}\{ M_E^{\textnormal{hi}} \geq n_2+1 - (1-s(1-q))(n_2+1) \} \\
=& \mathbb{P}\{ M_E^{\textnormal{hi}} \geq s(1-q)(n_2 +1) \} \\
\geq& 1-s.
\end{align*}

Similarly, $Q_{rq}^{\textrm{lo}}(E)$ is defined as the
$rq- (1-rq)/n_2$-th quantile of
$\{E^{\textnormal{lo}}_i:i\in\mathcal{I}_2\}$. This is equal to the
$\lceil n_2 (rq- (1-rq)/n_2)\rceil = \lceil (n_2+1) rq-1\rceil$
smallest value of 
$\{E^{\textnormal{lo}}_i:i\in\mathcal{I}_2\}$. Thus, if
$M_E^{\textrm{lo}}\geq \lceil (n_2+1) rq-1\rceil$, then
$R^{\textnormal{lo}}(X_{n+1})$ will be at least
$Q_{rq}^{\textrm{lo}}(E)$.  Then, our lower bound for the
$M_E^{\textrm{lo}}$ inverse CDF tells us that
\begin{align*}
\mathbb{P}\{ R^{\textnormal{lo}}(X_{n+1}) \geq  Q_{rq}^{\textnormal{lo}}(E)\}
\geq& \mathbb{P}\{ M_E^{\textnormal{lo}} \geq \lceil (n_2+1) rq-1\rceil\}\\
\geq& \mathbb{P}\{ M_E^{\textnormal{lo}} \geq (n_2+1)rq\} \\
\geq& 1-r.
\end{align*}

Finally, by the union bound, we have that
$$\mathbb{P}\{Q_{rq}^{\textnormal{lo}}(E)\leq R^{\textnormal{lo}}(X_{n+1})\textnormal{ and } R^{\textnormal{hi}}(X_{n+1})\leq Q_{1-s(1-q)}^{\textnormal{hi}}(E)\}\geq 1-r-s=1-\alpha$$
completing our proof.

\subsection{Impossibility of Capturing the Distribution Mean}
\label{Conditional Mean Proof}

Instead of proving the impossibility of capturing the conditional mean
of a distribution, we prove a more general result: we show that there
does not exist an algorithm to capture the mean of a distribution
$Y\sim P$ given no assumptions about $P$. This is a more general form
of our result because if we set $X\indep Y$ in $(X,Y)\sim P$, then
$\mathbb{E}[Y|X]=\mathbb{E}[Y]$, meaning that the impossibility of
capturing the mean results in the conditional mean being impossible to
capture as well.

Consider an algorithm $\hat{C}_n$ that, given i.i.d.~samples
$Y_1,\ldots,Y_n \sim P$, returns a (possibly randomized) confidence
interval $\hat{C}_n(\mathcal{D})$,
$\mathcal{D} = \{Y_i, 1 \le i \le n\}$, with length bounded by some
function of $P$ that captures $\mathbb{E}[Y]$ with probability at
least $1-\alpha$,
i.e.~$\mathbb{P}\{\mathbb{E}[Y]\in \hat{C}_n(\mathcal{D})\}\geq
1-\alpha$.
Pick $a>\alpha$, with $a<1$. Consider a distribution $P$ where for
$Y\sim P$, $\mathbb{P}\{Y=0\}=a^{1/n}$ and
$\mathbb{P}\{Y=u\}=1-a^{1/n}$ for some $u$. Then for
$Y_1,\ldots,Y_n\sim P$, $\mathbb{P}\{Y_1=\cdots=Y_n=0\}\geq a$.
Consider $\hat{C}_n(\{0,\ldots,0\})$; by our assumption on
$\hat{C}_n$, there must exist some $m\in\mathbb{R}$ for which
$\mathbb{P}\{m\in \hat{C}_n(\{0,\ldots,0\})\}<1-\alpha/a$. Then,
setting $u = \frac{m}{1-a^{1/n}}$ yields $\mathbb{E}[Y]=m$.  With
probability $a$, $\hat{C}_n(\mathcal{D}) = \hat{C}_n(\{0,\ldots,0\})$,
so
$$\mathbb{P}\{\mathbb{E}[Y]=m\not\in \hat{C}_n(\mathcal{D})\}> a\cdot\frac{\alpha}{a} = \alpha.$$
This implies that $\mathbb{P}\{\mathbb{E}[Y]\in \hat{C}_n(\mathcal{D})\}< 1-\alpha$
as desired, completing the proof.

\subsection{Capturing the Distribution Median}
\label{Median Proof}

\begin{algorithm}[H]
\caption{Confidence Interval for $\textrm{Median}(Y)$ with Coverage $1-\alpha$}
\label{Median Interval}
\SetKwInOut*{KwIn}{}
\SetKwInOut{KwProcess}{Process}
\SetKwInOut*{KwOut}{}
\KwIn{}

\begin{algorithmic}
\STATE Number of i.i.d.~datapoints $n\in\mathbb{N}$.
\STATE Datapoints $Y_1,\ldots,Y_n\sim P\subseteq\mathbb{R}$.
\STATE Coverage level $1-\alpha\in (0,1)$.
\end{algorithmic}

\KwProcess{}
\begin{algorithmic}
\STATE Order the $Y_i$ as $Y_{(1)}\leq \cdots \leq Y_{(n)} $.
\STATE Calculate the largest $k\geq 0$ such that for $X\sim \textrm{Binom}(n,0.5)$, we have $\mathbb{P}\{X < k\}\leq \alpha/2$. 
\end{algorithmic}

\KwOut{}
\begin{algorithmic}
\STATE Confidence interval $\hat{C}_n = [Y_{(k)},Y_{(n+1-k)}]$ for $\textrm{Median}(Y)$.
\STATE (Note that $Y_{(0)}=-\infty$ and $Y_{(n+1)}=\infty$)
\end{algorithmic}
\end{algorithm}

We now show that Algorithm \ref{Median Interval} captures the median
of $P$ with probability at least $1-\alpha$.  Let
$m = \textrm{Median}(P)$, let
$M^{\textnormal{lo}} = \#\{Y_i\leq m:1\leq i\leq n\}$ be the number of
$Y_i$ at most $m$, and let
$M^{\textnormal{hi}} = \#\{Y_i\geq m:1\leq i\leq n\}$ be the number of
$Y_i$ at least $m$. Note that by the definition of $m$, we have that
for all $i$, $\mathbb{P}\{Y_i\leq m\}\geq 0.5$, and
$\mathbb{P}\{Y_i\geq m\}\geq 0.5$ as well. Additionally, the events
$\{Y_i\leq m\}$ are mutually independent for all $i$, as are the events
$\{Y_i\geq m\}$. This implies that both $M^{\textnormal{lo}}$ and
$M^{\textnormal{hi}}$ follow a $\textrm{Binom}(n,0.5)$ distribution.

Since $\hat{C}_n = [Y_{(k)}, Y_{(n+1-k)}]$, we have that
$m\in\hat{C}_n$ if and only if $M^{\textnormal{lo}}\geq k$ and
$M^{\textnormal{hi}}\geq k$. Then,
\begin{align*}
    \mathbb{P}\{m\in \hat{C}_n\}
    =& \mathbb{P}\{M^{\textnormal{lo}}\geq k\textrm{ and }M^{\textnormal{hi}}\geq k\} \\
    =& 1 - \mathbb{P}\{M^{\textnormal{lo}}< k\textrm{ or }M^{\textnormal{hi}}< k\} \\
    \geq & 1 - (\mathbb{P}\{M^{\textnormal{lo}}< k\} + \mathbb{P}\{M^{\textnormal{hi}}< k\}) \\
    \geq & 1 - (\alpha/2 + \alpha/2) \\
    = & 1 - \alpha.
\end{align*}

\end{document}